\documentclass[12pt, a4paper]{elsarticle}
\usepackage{amsthm,amsmath,amssymb,
extsizes, mathabx}
\usepackage{color}

\usepackage[active]{srcltx}
\usepackage{MnSymbol}
\title{Miniversal
deformations of
pairs of \\ 
skew-symmetric
{matrices under congruence}}
\date{}
\author{Andrii Dmytryshyn} 
\ead{andrii@cs.umu.se}
\address{Department of Computing Science,
Ume{\aa} University, SE-901 87 Ume{\aa}, Sweden}

\DeclareMathOperator{\Ind}{Ind}

\newcommand{\sdotsss}%
{\text{\raisebox{-2.2pt}{$\cdot\,$}%
 \raisebox{1.7pt}{$\cdot$}%
\raisebox{5.6pt}{$\,\cdot$}}}
\newcommand{\hide}[1]{}

\renewcommand{\le}{\leqslant}
\renewcommand{\ge}{\geqslant}

\newtheorem{theorem}{Theorem}[section]
\newtheorem{lemma}{Lemma}[section]
\newtheorem{corollary}{Corollary}[section]

\theoremstyle{definition}
\newtheorem{definition}{Definition}[section]

\theoremstyle{remark}
\newtheorem{remark}{Remark}[section]
\newtheorem{example}{Example}[section]

\begin{document}
\begin{abstract}
Miniversal deformations for pairs of skew-symmetric matrices under congruence are constructed. To be precise, for each such a pair $(A,B)$ we provide a normal form with a minimal number of independent parameters to which all pairs of skew-symmetric matrices $(\widetilde{A},\widetilde{B})$, close to $(A,B)$ can be reduced by congruence transformation which smoothly depends on the entries of the matrices in the pair $(\widetilde{A},\widetilde{B})$. An upper bound on the distance from such a miniversal deformation to $(A,B)$ is derived too. We also present an example of using miniversal deformations for analyzing changes in the canonical structure information (i.e. eigenvalues and minimal indices) of skew-symmetric matrix pairs under perturbations.    
\end{abstract}

\maketitle

\section{Introduction}

Canonical forms of matrices and matrix pencils, e.g., Jordan and Kronekher canonical forms, are well known and studied with various purposes but the reductions to these forms are unstable operations: both the corresponding canonical forms and the reduction transformations depend discontinuously on the entries of an original matrix or matrix pencil. 
Therefore, V.I. Arnold introduced a normal form, with the minimal number of independent parameters, to which  
an arbitrary family of matrices $\tilde{A}$ 
close to a given matrix $A$ can be reduced by similarity transformations smoothly depending on the entries of $\tilde{A}$. 
He called such a normal form a miniversal deformation of $A$.
Now the notion of miniversal deformations has been extended to matrices with respect to congruence \cite{bilin} and *congruence \cite{sf}, matrix pencils with respect to strict equivalence \cite{bo1, sgp} and congruence \cite{dm2}, etc. (more detailed list is given in the introduction of \cite{sf}). 

Miniversal deformations can help us to construct stratifications, i.e., closure hierarchies, \cite{bfg, bo2, bo4} of orbits and bundles. These stratifications are the graphs that show which canonical forms the matrices (or matrix pencils) may have in an arbitrarily small neighbourhood of a given matrix (or matrix pencil). 
In particular, the stratifications show how the eigenvalues may coalesce or split apart, appear or disappear. Both the stratifications and miniversal deformations may be useful when the matrices arise as a result of measures and their entries are known with errors, see \cite{bobook,ref4} for some applications in control {and stability} theory. 

The questions related to eigenvalues and another canonical information for the pencils $A - s B$, where $A= \pm A^T$ and $B = \pm B^T$, or $A= \pm A^*$ and $B = \pm B^*$, 
dragged some attention over time and, especially, recently, e.g., see the following papers on canonical forms \cite{mo1, thom}, codimension computations \cite{dt_d_1, dt_d_2, dmsystss, dmsysts}, low rank perturbations \cite{b}, miniversal deformations \cite{dm2, bilin, sgp}, partial \cite{bfg, f_k_s} and general \cite{skewstr} stratification results, staircase forms \cite{mehr1, mehr2}. Such pencils also appear as the structure preserving linearizations of the corresponding matrix polynomials \cite{m_m_t, m4skew}.  
In particular, the papers \cite{b, skewstr, dmsystss, mo1, thom} deal with skew-symmetric matrix pencils, i.e. $A - s B$, where $A= - A^T$ and $B = - B^T$, and \cite{m4skew} deals with skew-symmetric matrix polynomials. 
{Skew-symmetric matrix pencils appear in multisymplectic partial differential equations \cite{multi}, systems with bi-Hamiltonian structure \cite{biham}, as well as in the design of a passive velocity field controller \cite{robot}.
Recall that, an $n \times n$ skew-symmetric matrix pencil $A - s B$ is called congruent to $C - s D$ if and only if there is a non-singular matrix $S$ such that $S^TAS = C$ and $S^TBS = D$. The set of matrix pencils congruent to a skew-symmetric matrix pencil $A - s B$ is called a congruence orbit of $A - s B$.}

In this paper, we derive the miniversal deformations of skew-symmetric matrix pencils {under congruence} and bound the distance from these deformations to unperturbed matrix pencils in terms of the norm of the perturbations. The number of independent parameters in the miniversal deformations is equal to the codimensions of the congruence orbits of skew-symmetric matrix pencils (obtained independently in \cite{dmsystss}). The Matlab functions for computing these codimensions were developed  \cite{sstool} and added to the Matrix Canonical Structure (MCS) Toolbox \cite{toolbox}. Example \ref{exskew} shows how the miniversal deformations from Theorem \ref{teo2} can be used for the investigation of the possible changes of the canonical structure information.

The rest of the paper is organized as follows. 
In Section $2,$ we present the main theorems, i.e., we construct miniversal deformations of skew-symmetric matrix pencils and prove an upper bound on the distance between a skew-symmetric matrix pencil and its miniversal deformation. 
In Section $3$, 
we describe the method of constructing deformations (Section $3.1$) and derive the miniversal deformations step by step: 
for the diagonal blocks (Section $3.2$), 
for the off-diagonal blocks that correspond to the canonical summands of the same type (Section $3.3$), 
and for the off-diagonal blocks that correspond to the canonical summands of different types (Section $3.4$).

In this paper all matrices are considered over the field of complex numbers. 
Except in Example \ref{exskew}, we use the matrix pair notation $(A,B)$ instead of the pencil notation $A - s B$. We also use one calligraphic letter, e.g., ${\cal A }$ or ${\cal D}$, to refer to a matrix pair.  

\section{The main results}

In this section, we present the miniversal deformations of pairs
of skew-symmetric matrices under congruence 
and obtain an upper bound on the distance between a skew-symmetric matrix pair and its miniversal deformations. 
In Section \ref{sur}, we will derive these miniversal deformations. 

First we recall the canonical form of pairs of skew-symmetric
matrices under congruence given in \cite{thom}.
For each $k=1,2, \dots$, define the $k\times k$
matrices
\begin{equation*}\label{1aa}
J_k(\lambda):=\begin{bmatrix}
\lambda&1&&\\
&\lambda&\ddots&\\
&&\ddots&1\\
&&&\lambda
\end{bmatrix},\qquad
I_k:=\begin{bmatrix}
1&&&\\
&1&&\\
&&\ddots&\\
&&&1
\end{bmatrix},
\end{equation*}
where $\lambda \in \mathbb C$, and for each $k=0,1, \dots$, the $k\times
(k+1)$ matrices
\begin{equation*}
F_k :=
\begin{bmatrix}
1&0&&\\
&\ddots&\ddots&\\
&&1&0\\
\end{bmatrix}, \qquad
G_k :=
\begin{bmatrix}
0&1&&\\
&\ddots&\ddots&\\
&&0&1\\
\end{bmatrix}.
\end{equation*}
All non-specified entries of the matrices $J_k(\lambda), I_k,F_k,$ and $G_k$ are zero. 
\begin{lemma}[\cite{s_76, thom}]\label{lkh}
Every pair of
skew-symmetric complex
matrices is congruent
to a direct sum,
determined uniquely up
to permutation of
summands, of pairs of
the form
\begin{align}
\label{can}
\mathcal H_n(\lambda)&:= \left(
\begin{bmatrix}0&I_n\\
-I_n &0
\end{bmatrix},
\begin{bmatrix}0&J_n(\lambda)\\
-J_n(\lambda)^T &0
\end{bmatrix}
 \right),\quad \lambda \in\mathbb C,\\
\label{can2}
\mathcal K_n&:= \left(
\begin{bmatrix}0&J_n(0)\\
-J_n(0)^T &0
\end{bmatrix},
\begin{bmatrix}0&I_n\\
-I_n&0
\end{bmatrix}
 \right),\\
\label{can3}
\mathcal L_n&:= \left(
\begin{bmatrix}0&F_n\\
-F_n^T &0
\end{bmatrix},
\begin{bmatrix}0&G_n\\
-G_n^T&0
\end{bmatrix}
 \right).
  \end{align}
\end{lemma}
\noindent Thus, each pair of skew-symmetric
matrices is congruent to a direct sum
\begin{equation}\label{kus}
(A,B)_{\rm can}=
\bigoplus_{i=1}^{a} \mathcal  H_{h_i}(\lambda_i)
 \oplus
\bigoplus_{j=1}^{b} \mathcal K_{k_j}
 \oplus
\bigoplus_{r=1}^{c} \mathcal L_{l_r},
\end{equation}
consisting of direct summands of three
types {of pairs}.

\subsection{Miniversal deformations}

The concept of a
miniversal deformation
of a matrix with
respect to similarity was
given by V. I. Arnold
\cite{arn} (see also
\cite[\S\,30B]{arn3}).
This concept can straightforwardly be extended to pairs of skew-symmetric matrices with respect to congruence.

A \emph{deformation}
of a pair of skew-symmetric $\hat{n}\times \hat{n}$ matrices $(A,B)$ is
a holomorphic mapping
${\cal A}(\vec\delta)$, where $\vec\delta=(\delta_1,\dots, \delta_k)$, from a
neighborhood
$\Omega \subset
\mathbb C^k$ of $\vec
0=(0,\dots,0)$ to the
space of pairs of
skew-symmetric $\hat{n}\times \hat{n}$ matrices
such that ${\cal A}(\vec 0)=(A,B)$. Note that in this paper we consider only skew-symmetric deformations, i.e., the skew-symmetric structure of matrix pairs is preserved. Therefore we write only ``deformation'' but not ``skew-symmetric deformation'' without the risk of confusion.

\begin{definition}\label{d}
A deformation ${\cal
A}(\delta_1,\dots,\delta_k)$
of a pair of skew-symmetric matrices 
$(A,B)$ is called
\emph{versal} if for every
deformation ${\cal
B}(\sigma_1,\dots,\sigma_l)$
of $(A,B)$ we have  
\begin{equation*}\label{kft}
{\cal B}(\sigma_1,\dots,\sigma_l)=
I(\sigma_1,\dots,\sigma_l)^{T}
{\cal
A}(\varphi_1(\vec\sigma),\dots,
\varphi_k(\vec\sigma))
I(\sigma_1,\dots,\sigma_l),
\end{equation*}
where $I(\sigma_1,\dots,\sigma_l)$ is a deformation of the identity matrix, 
and all
$\varphi_i(\vec\sigma)$
are convergent in a
neighborhood of $\vec
0$ power series such
that $\varphi_i(\vec
0)=0$. A versal
deformation ${\cal
A}(\delta_1,\dots,\delta_k)$
of $(A,B)$  is called
\emph{miniversal} if
there is no versal
deformation having
less than $k$
parameters.
\end{definition}

By a \emph{$(0,*)$
matrix} we mean a
matrix whose entries
are $0$ and $*$ and we consider pairs $\cal D$ of $(0,*)$
matrices. We
say that a pair of skew-symmetric matrices
\emph{is of the form}
$\cal D$ if it can be
obtained from $\cal D$
by replacing the stars
with complex numbers, respecting the skew-symmetry.
Denote by ${\cal
D}({\mathbb C})$ the
space of all pairs of skew-symmetric matrices 
of the form
$\cal D$, and by
${\cal D}(\vec
{\varepsilon})$ the
pair of parametric skew-symmetric 
matrices obtained from
${\cal D}$ by
replacing the $(i,j)$-th and $(j,i)$-th
stars with the parameters
${\varepsilon}_{ij}$ and $-{\varepsilon}_{ji}$, respectively, 
in the first matrix and the $(i',j')$-th and $(j',i')$-th
stars with the parameters ${\varepsilon}^{'}_{i'j'}$ and $-{\varepsilon}^{'}_{j'i'}$, respectively, 
in the second matrix. In other words 
\begin{equation} \label{a2z}
{\cal D}(\vec
{\varepsilon}):=
\Big(\sum_{(i,j) \in  
\Ind_1({\cal D})}
\varepsilon_{ij}E_{ij},
\sum_{(i',j')\in \Ind_2({\cal D})}
{\varepsilon}^{'}_{i'j'}E_{i'j'}\Big),
\end{equation}
\begin{equation} \label{spacedc}
{\cal D}(\mathbb C):= 
\left\{
{\cal D}(\vec{\varepsilon}) \ | \ \vec{\varepsilon} \in \mathbb C^k
\right\} = 
\left\{
\Big(\bigplus_{(i,j)\in 
\Ind_1({\cal D})}
{\mathbb C}
E_{ij}, 
\bigplus_{(i',j')\in
\Ind_2({\cal D})}
{\mathbb C}
E_{i'j'}\Big)
\right\},
\end{equation}
where
\begin{equation*}\label{a2za}
\Ind_1({\cal
D}),\Ind_2({\cal
D})\subseteq
\{1,\dots, \hat{n}\}\times
\{1,\dots, \hat{n}\},
\end{equation*}
are the sets of
indices of the stars
in the upper-triangular parts of the first and
the second matrices,
respectively, of the
pair ${\cal D}$, and
{$E_{ij}$ is the matrix
whose $(i,j)$-th entry
is $1$, $(j,i)$-th entry
is $-1$ and the other entries 
are zero}. Note that the large ``$\bigplus$'' in \eqref{spacedc} denotes 
the entrywise sum of matrices.

Following \cite{sgp}, we say that a
miniversal deformation
of $(A,B)$ is
\emph{simplest} if it
has the form
$(A,B)+{\cal D}(\vec
{\varepsilon})$, where
$\cal D$ is a pair of $(0,*)$
matrices. If the matrix pair $\cal
D$ in $(A,B)+{\cal D}(\vec
{\varepsilon})$ has no zero
entries (except on the main diagonals), 
then $\cal D$ defines the deformation 
\begin{equation}\label{edr}
{\cal U}(\vec
{\varepsilon}):=
\Big(A+\sum_{i=1}^{\hat{n}} \sum_{j=i+1}^{\hat{n}}
\varepsilon_{ij}E_{ij},\
B+\sum_{i=1}^{\hat{n}} \sum_{j=i+1}^{\hat{n}}
{\varepsilon}^{'}_{ij}E_{ij}\Big).
\end{equation}

In other words, for all pairs of $\hat{n} \times \hat{n}$ skew-symmetric matrices
$(A+E,B+E')$ that are close to a given
pair of skew-symmetric
matrices $(A,B)$,
we derive the normal form ${\cal A}(E,E')$ with respect to the congruence
transformation
\begin{equation} \label{trans}
(A+E,B+E') \mapsto S(E,E')^T(A+E,B+E') S(E,E') =: {\cal A}(E,E'),
\end{equation}
in which $S(E,E')$ is holomorphic at 0 (i.e. its
entries are power series in the entries of $E$ and $E'$
that are convergent in a neighborhood of 0) and
$S(0,0)$ is a nonsingular $\hat{n} \times \hat{n}$ matrix.

Since ${\cal A}(0,0)=S(0,0)^T(A,B)S(0,0)$, we can
take ${\cal A}(0,0)$ equal to the
congruence canonical
form $(A,B)_{\text{\rm
can}}$ of $(A,B)$, see \eqref{kus}.
Then
\begin{equation}
\label{ksy} {\cal
A}(E,E')=
(A,B)_{\text{\rm
can}}+{\cal D}(E,E'),
\end{equation}
where  ${\cal
D}(E,E')$ ($={\cal D}(\vec{\varepsilon}) $ for some $ \vec{\varepsilon} \in \mathbb C^k$) is a pair of
skew-symmetric matrices that is
holomorphic at $0$ and
${\cal D}(0,0)=(0,0)$.
In Theorem \ref{teo2} we
derive ${\cal
D}(E,E')$ with the
minimal number of
nonzero entries that
can be attained using the congruence transformation defined in 
\eqref{trans}.

We use the following
notation, in
which each star
denotes a function of
the entries of $E$ and
$E'$ that is
holomorphic at zero:

$\bullet$ $0_{mn}$ is
the $m \times n$ zero
matrix;

$\bullet$
$0_{mn \ast}$
is the $m \times n$
matrix
$
\begin{bmatrix}
&  && 0\\
&0_{m-1,n-1}&& \vdots\\
&  && 0\\
0&\ldots&0&*
\end{bmatrix};
$

$\bullet$
$0_{mn}^{\nwarrow}$
is the $m \times n$
matrix
\begin{equation}\label{mgde}
\begin{bmatrix}
\begin{matrix}
*\\ *\\\vdots\\ *
\end{matrix}
&0_{m,n-1}
\end{bmatrix}\quad
 \text{if $m\le n$, and }
\begin{bmatrix}
  *\ *\ \dots\ *
 \\[2mm]
   0_{m-1,n}
\\[2mm]
\end{bmatrix}\quad
\text{if $m\ge n$}
 \end{equation}
(if $m=n$, then we can
take any of the
matrices defined in
\eqref{mgde});

$\bullet$
$0^{\nearrow}$,
$0^{\searrow}$
and $0^{\swarrow}$
are matrices that
are obtained from
$0^{\nwarrow}$,
by clockwise
rotation by $90^{\circ}$,
$180^{\circ}$ and
$270^{\circ}$, respectively;

$\bullet$
$0_{mn}^{\leftarrow}$
is the
$m \times n$ matrix
$\label{bjhf}
\begin{bmatrix}
  * &  \\
   \vdots & 0_{m,n-1} \\
  * &
\end{bmatrix}
$ 

{(in contrast to $0_{mn}^{\nwarrow}$ and $0_{mn}^{\swarrow}$, the matrix $0_{mn}^{\leftarrow}$ has stars in the first column even if $m>n$);}

$\bullet$
$0_{mn}^{\rightarrow}$ is the
$m \times n$ matrix  $
\begin{bmatrix}
  & *  \\
  0_{m,n-1} & \vdots \\
  & *
\end{bmatrix}
$

{(in contrast to $0_{mn}^{\nearrow}$ and $0_{mn}^{\searrow}$,  the matrix $0_{mn}^{\rightarrow}$ has stars in the last column even if $m>n$);}

{
$\bullet$
$0_{mn}^{\righthalfcap}$
is the $m \times n$
matrix
$
 \begin{bmatrix}
* & & \\
\vdots&0_{m-1,n-1}& \\
*&\ldots & *
\end{bmatrix}; 
$}

$\bullet$ $0^{\boxminus}_{mn}$ with $m < n$
is the $m \times n$
matrix
\begin{equation*}\label{hui}
\begin{bmatrix}
\begin{matrix}
0&\dots& 0\\
\vdots&& \vdots
\end{matrix} &0\\
\begin{matrix}
0& \dots&0\end{matrix}&
\begin{matrix}
 *\ \dots\ * \ 0\
\end{matrix}
\end{bmatrix}\qquad(\text{$n-m$
stars})
\end{equation*}
if $m\geq n$ then $0^{\boxminus}_{mn}=0$.

Further, we will usually omit the indices $m$ and $n$.

{
Let
\begin{equation}\label{gto}
(A,B)_{\text{\rm \rm
can}}={\mathcal X}_1\oplus\dots
\oplus {\mathcal X}_t
\end{equation}
be a canonical pair of
skew-symmetric complex
matrices for
congruence, in which
${\mathcal X}_1,\dots,{\mathcal X}_t$ are
pairs of the form
\eqref{can}--\eqref{can3}, and let 
${\cal D}(E,E')$ be a
pair of skew-symmetric matrices, defined in \eqref{ksy}, whose
matrices are
partitioned into
blocks conformally to
the decomposition
\eqref{gto}:
\begin{equation}\label{grsd}
{\cal D}(E,E')= {\cal D} = \left(
\begin{bmatrix}
D_{11}&\dots&
D_{1t}
 \\
\vdots&\ddots&\vdots\\
D_{t1}&\dots&
D_{tt}
\end{bmatrix},
\begin{bmatrix}
D'_{11}&\dots&
D'_{1t}
 \\
\vdots&\ddots&\vdots\\
D'_{t1}&\dots&
D'_{tt}
\end{bmatrix}
\right). 
\end{equation}
Note that $(D_{ji},D'_{ji})=(-D_{ij}^T,-D_{ij}^{'T})$ and 
define
\begin{equation}\label{lhsd}
{\cal D}(\mathcal X_i):=(D_{ii},D'_{ii}) \quad  \text{ and } \quad 
 {\cal
D}(\mathcal X_i, \mathcal X_j):=(
D_{ij},D'_{ij}), \ i<j.
\end{equation}}

Since each pair of skew-symmetric matrices is congruent to its
canonical pair of matrices, it suffices to construct the miniversal
deformations for the pairs of canonical matrices
(i.e. direct sums of the pairs \eqref{can}--\eqref{can3}). 

\begin{theorem}\label{teo2}
Let $(A,B)_{\text{\rm \rm can}}$
be a pair of
skew-symmetric complex
matrices 
\eqref{kus}. 
A simplest
miniversal deformation of $(A,B)_{\text{\rm can}}$
can be taken in the
form $(A,B)_{\text{\rm
can}} +{\cal D}$ in which ${\cal D}$ is a
pair of $(0,*)$ matrices (the stars denote
independent parameters, up to
skew-symmetry, see also Remark \ref{indep}) whose
matrices are partitioned into blocks conformally to
the decomposition of $(A,B)_{\text{\rm can}}$, see \eqref{grsd}, and 
the blocks of ${\cal D}$ are
defined, in the notation \eqref{lhsd}, as follows:  

{\rm(i)} The
{\rm diagonal blocks} of
${\cal D}$ are defined
by
\begin{align}
\label{Hdef}
{\cal
D}(\mathcal H_n(\lambda))&=
\left( 0,
\begin{bmatrix}
0&0^{\swarrow}\\
0^{\nearrow}&0
\end{bmatrix} \right),\\
\label{Kdef}
{\cal D} (\mathcal K_n)&=\left(
\begin{bmatrix}
0&0^{\swarrow}\\
0^{\nearrow}&0
\end{bmatrix}, 0 \right),\\
\label{Ldef}
{\cal D}(\mathcal L_n)&=(0,0).
\end{align}

{\rm(ii)} The
{\rm off-diagonal blocks} of
${\cal D}$ whose
horizontal and
vertical strips
contain pairs of
$(A,B)_{\text{\rm
can}}$ of the {\rm same
type} are defined by
\begin{align}\label{lsiu1}
{\cal D}
(\mathcal H_n(\lambda), \mathcal H_m(\mu))&
       =
  \begin{cases}
(0,\:0) &\text{if
$\lambda\ne\mu,$}
      \\
    \left(0,\: \begin{bmatrix}
0^{\searrow}&0^{\swarrow}
 \\ 0^{\nearrow}&0^{\nwarrow}
\end{bmatrix}
\right)
 &\text{if $\lambda=\mu,$}
  \end{cases} \\
\label{lsiu2}
{\cal D} (\mathcal K_n,\mathcal K_m)&=
  \left(
 \begin{bmatrix}
0^{\searrow}&0^{\swarrow}
 \\ 0^{\nearrow}&0^{\nwarrow}
\end{bmatrix} , 0
\right),\\
\label{lsiu3}
{\cal D}
(\mathcal L_n,\mathcal L_m)&=\left(
 \begin{bmatrix}
0&0
 \\ 0&0_{\ast}
\end{bmatrix} ,
 \begin{bmatrix}
0&0^{\boxminus T}_{m+1,n}
 \\ 0^{\boxminus}_{n+1,m}&0^{\righthalfcap}
\end{bmatrix}
\right).
\end{align}
{\rm(iii)} The
{\rm off-diagonal blocks} of
${\cal D}$ whose
horizontal and
vertical strips
contain pairs of
$(A,B)_{\text{\rm
can}}$ of {\rm different
types} are defined by
\begin{align}\label{kut}
{\cal D}
(\mathcal H_n(\lambda),\mathcal K_m)&=(0,0),\\
\label{hnlm}
{\cal D}
(\mathcal H_n(\lambda),\mathcal L_m)&=\left(
0,
 \begin{bmatrix}
0&0^{\leftarrow}
\end{bmatrix}
\right),\\
\label{ktlm}
 {\cal D}
(\mathcal K_n,\mathcal L_m)&=\left(0^{\rightarrow} ,
0
\right).
\end{align}
\end{theorem}

\begin{remark}[About the independency of parameters] \label{indep}
All parameters placed instead of the stars in the upper triangular parts of matrices of ${\cal D}$ are independent and the lower triangular parts are defined by the skew-symmetry. In particular, it means that parametric matrix pairs obtained from $(D_{ij},  D'_{ij})$ and $(D_{i'j'}, D'_{i'j'})$ have dependent (in fact, equal up to the sign) parametric entries if and only if $i'=j$ and $j'=i$.
\end{remark}

Let us give an example of how the miniversal deformations from Theorem \ref{teo2} can be used for the investigation of changes of the canonical structure information under small perturbations. \begin{example} \label{exskew}
We show that in an arbitrarily small neighbourhood of a matrix pair with the canonical form $\mathcal L_1 \oplus \mathcal L_0$ there is always a matrix pair with the canonical form $\mathcal  H_2(\lambda), \lambda \neq 0$ (in fact, also with $\mathcal  H_2(0)$ and $\mathcal  K_2$). 

It is enough to consider perturbations of $\mathcal L_1 \oplus \mathcal L_0$ in the form of the miniversal deformations given in Theorem \ref{teo2} (with only three independent nonzero parameters). Since we will use the theory developed for matrix pencils we switch to the pencil notation $X - s Y$, instead of $(X,Y)$. Thus a miniversal deformation of $\mathcal L_1 \oplus \mathcal L_0$ is the pencil  
{\footnotesize
\begin{equation} \label{1pen}
\begin{bmatrix} 0&1&0&0\\
-1&0&0&0\\
0&0&0&\varepsilon_1\\
0&0&-\varepsilon_1&0
 \end{bmatrix} - s 
 \begin{bmatrix} 0&0&1&0\\
0&0&0&\varepsilon_2\\
-1&0&0&\varepsilon_3\\
0&-\varepsilon_2&-\varepsilon_3&0
 \end{bmatrix}
 =
\begin{bmatrix} 0&1&- s&0\\
-1&0&0&- s \varepsilon_2\\
 s&0&0&\varepsilon_1 - s \varepsilon_3\\
0& s \varepsilon_2&-\varepsilon_1 + s \varepsilon_3&0
 \end{bmatrix},
 \end{equation}
 }
 which has the Smith form (see \cite{maltcev} for the definition)
 {\small
\begin{equation} \label{1smith}
 \begin{bmatrix} 1&0&0&0\\
0&1&0&0\\
0&0&\varepsilon_1-s \varepsilon_2 - s^2 \varepsilon_3 &0\\
0&0&0&\varepsilon_1-s \varepsilon_2 - s^2 \varepsilon_3
 \end{bmatrix}.
\end{equation}
}
In turn, the pencil $\mathcal H_2(\lambda)$ is  
{\footnotesize
\begin{equation} \label{2pen}
\begin{bmatrix} 0&0&1&0\\
0&0&0&1\\
-1&0&0&0\\
0&-1&0&0
 \end{bmatrix} - s 
 \begin{bmatrix} 0&0&\lambda&1\\
0&0&0&\lambda\\
-\lambda&0&0&0\\
-1&-\lambda&0&0
 \end{bmatrix} = \begin{bmatrix} 0&0&1-s \lambda&s \\
0&0&0&1 - s \lambda\\
-1+ s \lambda&0&0&0\\
-s &-1+ s \lambda&0&0
 \end{bmatrix},
\end{equation}}
and has the Smith form 
{\small
\begin{equation} \label{2smith}
 \begin{bmatrix} 1&0&0&0\\
0&1&0&0\\
0&0&1- 2 s \lambda - s^2 \lambda^2 &0\\
0&0&0&1- 2 s \lambda - s^2 \lambda^2
 \end{bmatrix}.
\end{equation}}\noindent 
Now \eqref{1smith} with $\varepsilon_2= 2 \varepsilon_1 \lambda$ and $ \varepsilon_3 = \varepsilon_1 \lambda^2$ is strictly equivalent to \eqref{2smith} which implies that the pencils \eqref{1pen} and \eqref{2pen} are strictly equivalent by \cite[Proposition A.5.1, p. 663]{lancrod} (note that $\lambda \neq 0$ and we must choose $\varepsilon_1 \neq 0$) and due to \cite[Theorem 3, p. 275]{maltcev} the pencils \eqref{1pen} and \eqref{2pen} are congruent. Since $\varepsilon_1$ (and thus $\varepsilon_2$ and $\varepsilon_3$) can be chosen arbitrarily small we can find a pair with the canonical form $\mathcal H_2(\lambda), \lambda \neq 0$ in any neighbourhood of $\mathcal L_1 \oplus \mathcal L_0$.

Note that, for \eqref{1pen} and \eqref{2pen} we could have also computed the skew-symmetric Smith form derived in \cite{m4skew}.  
The result of this example also follows from the more general result in \cite{skewstr} but the proof given here is constructive, i.e. the perturbation is derived explicitly.    
\end{example}

The pair of matrices $\cal
D$ $\eqref{grsd}$ in Theorem
\ref{teo2} will be
constructed in Section
\ref{sur} as follows.
The vector space
\begin{equation*}\label{msi}
T_{(A,B)_{\text{\rm
can}}}:=\{C^T(A,B)_{\text{\rm
can}}
+(A,B)_{\text{\rm
can}}C\,|\,
C\in{\mathbb
C}^{\hat{n}\times \hat{n}}\}
\end{equation*}
is the tangent space
to the congruence
class of
$(A,B)_{\text{\rm
can}}$ at the point
$(A,B)_{\text{\rm
can}}$ since
\begin{equation} \label{tangatp}
\begin{split}
(I+\varepsilon
C)^T(A,B)_{\text{\rm
can}} (I+\varepsilon
C) &=(A,B)_{\text{\rm
can}}+ \varepsilon(C^T
(A,B)_{\text{\rm
can}}+
(A,B)_{\text{\rm
can}}C) \\
&+\varepsilon^2C^T
(A,B)_{\text{\rm
can}}C
\end{split}
\end{equation}
for all ${\hat{n}}$-by-${\hat{n}}$
matrices $C$ and each
$\varepsilon\in\mathbb
C$. Then $\cal D$ is constructed such that 
\begin{equation}\label{jyr}
{\mathbb C}^{\,\hat{n}\times \hat{n}}_{c}\times{\mathbb
C}^{\,\hat{n}\times \hat{n}}_{c}=T_{(A,B)_{\text{\rm
can}}}
+ {\cal
D}({\mathbb C})
\end{equation}
in which ${\mathbb
C}^{\,\hat{n}\times \hat{n}}_{c}$
is the space of all skew-symmetric ${\hat{n}} \times {\hat{n}}$
matrices,  ${\cal
D}({\mathbb C})$ is
the vector space of
all pairs of skew-symmetric matrices obtained from $\cal D$
by replacing its stars by complex numbers, see \eqref{spacedc}.
Thus, one half of the number of
stars in $\cal D$ is
equal to the
codimension of the
congruence orbit of
$(A,B)_{\text{\rm
can}}$ (note that the total
number of the stars is always even).
Lemma \ref{t2.1} in Section \ref{31section} ensures
that any pair of $(0,*)$ matrices that satisfies
\eqref{jyr} can be
taken as $\cal D$ in
Theorem \ref{teo2}.

\subsection{Upper bound for the norm of miniversal deformations}
In this section, we bound the distance from the miniversal deformations to a matrix pair that was originally perturbed, using the norm of the perturbations. 
In particular, we see that this distance can be made arbitrarily small 
by decreasing the size of the allowed perturbations. Similar techniques are used in \cite{bilin,sf} to prove the versality of the deformations.

We use the Frobenius norm of a complex
$n\times n$ matrix $Y=[y_{ij}]$:
\begin{equation*}\label{4a}
\|Y\|:=\sqrt{\sum |y_{ij}|^2}.
\end{equation*}
Recall that for matrices $Y$ and
$Z$ and $\nu, \omega \in\mathbb
C$ the following inequalities hold (e.g., see \cite[Section
5.6]{hor_John})
\begin{equation}\label{lk}
\|\nu Y+ \omega Z\|\le
|\nu |\,\|Y\|+|  \omega |\,\|Z\| \quad \text{and} \quad
\|YZ\|\le \|Y\|\,\|Z\|.
\end{equation}
Let $(A,B)\in({\mathbb
C}^{\,\hat{n}\times \hat{n}}_c,{\mathbb
C}^{\,\hat{n}\times \hat{n}}_c)$ and $\alpha:=\|A\| , \beta:=\|B\|$. 
By
\eqref{jyr}, for each pair of skew-symmetric ${\hat{n}}$-by-${\hat{n}}$ matrices
$(E_{ij},0)$ and $(0,E_{i'j'})$, $1 \le i,j,i',j' \le {\hat{n}}$ there exist $X_{ij},X'_{i'j'}\in\mathbb
C^{\hat{n}\times \hat{n}}$ such
that
\begin{equation}
\begin{aligned}\label{8}
(E_{ij},0)&+X_{ij}^T(A+M,B+N)
+(A+M,B+N)X_{ij} \in {\cal
D}({\mathbb C}),\\
(0,E_{i'j'})&+X_{i'j'}^{'T}(A+M,B+N)
+(A+M,B+N)X'_{i'j'} \in {\cal
D}({\mathbb C}),
\end{aligned}
\end{equation}
where ${\cal
D}({\mathbb C})$ is
defined in \eqref{spacedc}.
If $(i,j)\in \Ind_1({\cal D})$, then
$(E_{ij},0) \in {\cal
D}({\mathbb C})$, and
so we can put
$X_{ij}=0$. Analogously, if $(i',j')\in \Ind_2({\cal D})$, then
$(E_{i'j'},0) \in {\cal
D}({\mathbb C})$, and
so we can put
$X_{i'j'}=0$.
Denote
\begin{equation}\label{gamma}
\gamma:=
\sum_{(i,j)\notin \Ind_1({\cal D})}
\|X_{ij}\|+ \sum_{(i',j')\notin{ \Ind_2({\cal D})}} \|X^{'}_{i'j'}\|.
\end{equation}

\begin{theorem} 
Let $(A,B)\in({\mathbb
C}^{\,\hat{n}\times \hat{n}}_c,{\mathbb
C}^{\,\hat{n}\times \hat{n}}_c)$ and
let $\varepsilon \in \mathbb{R}$ such that 
$$
0< \varepsilon < \frac{1}{\max \{ 1+\gamma(\alpha+1)(2+ \gamma),1+\gamma(\beta+1)(2+ \gamma)\} }, 
$$ 
{where $\alpha:=\|A\| , \beta:=\|B\|$  and $\gamma$ is defined in \eqref{gamma}}.
For each pair of skew-symmetric ${\hat{n}}$-by-${\hat{n}}$ matrices $(M,N)$ satisfying
\begin{equation} \label{15}
\|M\|<\varepsilon^{2},\qquad
\|N\|<\varepsilon^{2},
\end{equation}
there exists a matrix $S=I_{\hat{n}}+X$
depending holomorphically on the
entries of $(M,N)$ in a neighborhood of
zero such that
\[S^{T}(A+M,B+N)S=(A+P,B+Q), \ \ (P,Q) \in {\cal D}({\mathbb C}), \ \|  P \|<\varepsilon, \text{and} \ \| Q \|<\varepsilon, \] 
where ${\mathbb C}^{\,\hat{n}\times \hat{n}}_{c}\times{\mathbb C}^{\,\hat{n}\times \hat{n}}_{c}=T_{(A,B)_{\text{\rm can}}} + {\cal D}({\mathbb C})$.
\end{theorem}
\begin{proof}
First, note that if $M=0$ and $N=0$ then $S=I_{\hat{n}}$.

We construct
$S=I_{\hat{n}}+X$.  If
$M=\sum_{i,j}
m_{ij}E_{ij}$ and $N=\sum_{i,j}
n_{ij}E_{ij}$ (i.e.,
 $M=[m_{ij}]$ and $N=[n_{ij}]$),
then we can chose $X_{ij}$ and $X'_{ij}$ \eqref{8}, such that 
\begin{multline*}
\sum_{i,j}
(m_{ij}E_{ij},n_{ij}E_{ij})+\sum_{i,j}
(m_{ij}X_{ij}^{T}+n_{ij}X_{ij}^{'T})(A+M,B+N) \\ +
(A+M,B+N)\sum_{i,j}(m_{ij}X_{ij}+n_{ij}X'_{ij}) \in
{\cal D}({\mathbb C})
\end{multline*}
and for
\[
X:=\sum_{i,j}
(m_{ij}X_{ij}+n_{ij}X'_{ij})
\]
we have
\begin{equation*}\label{18}
(M,N)+X^T(A+M,B+N)+(A+M,B+N)X\in
{\cal D}({\mathbb C}).
\end{equation*}

If
$(i,j)\notin \Ind_1({\cal D})$ (or, respectively, $(i,j)\notin \Ind_2({\cal D})$), then
$|m_{ij}|<\varepsilon^{2}$ (or, respectively, $|n_{ij}|<\varepsilon^{2}$)
by \eqref{15}.
We obtain
\begin{align*}
\|X\|&\le
\sum_{(i,j)\notin \Ind_1({\cal D})}
|m_{ij}|\|X_{ij}\|+ \sum_{(i,j)\notin{ \Ind_2({\cal D})}} |n_{ij}|\|X^{'}_{ij}\| \\
&<
\sum_{(i,j)\notin \Ind_1({\cal D})}
\varepsilon^{2}\|X_{ij}\|+ \sum_{(i,j)\notin{ \Ind_2({\cal D})}} \varepsilon^{2} \|X^{'}_{ij}\|=
\varepsilon^{2} \gamma.
\end{align*}
Put
\[
S^{T}(A+M,B+N)S=(A+P,B+Q)\quad \text{where }
S:=I_{\hat{n}}+X,
\]
then
\begin{multline*}\label{18de}
(P,Q)=(M,N)+X^T(A+M,B+N)+(A+M,B+N)X \\
+X^T(A+M,B+N)X.
\end{multline*}
Summing up, we obtain 
\begin{align*}
\|P\|
&\le\|M\|+2\|X\|(\|A\|+\|M\|)
+\|X\|^2(\|A\|+\|M\|)
 \\&<
\varepsilon^{2}+2\varepsilon^{2}
\gamma(\alpha+\varepsilon^{2})+
\varepsilon^{4}\gamma^2(\alpha+\varepsilon^{2})
=\varepsilon^{2}+\varepsilon^{2}
\gamma(\alpha+\varepsilon^{2})(2+
\varepsilon^{2} \gamma)
 \\&<
\varepsilon^{2}(1+\gamma(\alpha+1)(2+ \gamma)) < \varepsilon,\\
\|Q\|
&\le\|N\|+2\|X\|(\|B\|+\|N\|)
+\|X\|^2(\|B\|+\|N\|)
 \\&<
\varepsilon^{2}(1+\gamma(\beta+1)(2+ \gamma)) < \varepsilon.
\end{align*}
\end{proof}

\section{Proof of the main theorem}
\label{sur}

\subsection{A method
of construction of
miniversal
deformations} \label{31section}

We give a method
of construction of
simplest miniversal
deformations, which
will be used in the
proof of Theorem
\ref{teo2}.

The deformation
\eqref{edr} is
universal in the sense
that every deformation
${\cal
B}(\sigma_1,\dots,\sigma_l)$
of $(A,B)$ has the form
${\cal
U}(\vec{\varphi}
(\sigma_1,\dots,\sigma_l)),$
where
$\varphi_{ij}(\sigma_1,\dots,\sigma_l)$
are convergent in a
neighborhood of $\vec
0$ power series such
that
$\varphi_{ij}(\vec 0)=
0$. Hence every
deformation ${\cal
B}(\sigma_1,\dots,\sigma_l)$
in Definition \ref{d}
can be replaced by
${\cal U}(\vec
{\varepsilon})$, which
proves the following
lemma.

\begin{lemma}\label{lem}
The following two
conditions are
equivalent for any
deformation ${\cal
A}(\delta_1,\dots,\delta_k)$
of pair of matrices $(A,B)$:
\begin{itemize}
  \item[\rm(i)]
The deformation ${\cal
A}(\delta_1,\dots,\delta_k)$
is versal.
  \item[\rm(ii)]
The deformation
\eqref{edr} is
equivalent to ${\cal
A}(\varphi_1(\vec{\varepsilon}),\dots,
\varphi_k(\vec{\varepsilon}))$
in which all
$\varphi_i(\vec{\varepsilon})$
are convergent in a
neighborhood of\/
$\vec 0$ power series
such that
$\varphi_i(\vec 0)=0$.
\end{itemize}
\end{lemma}

If $\rm U$ is a subspace
of a vector space $\rm V$,
then each set $v+\rm U$
with $v\in \rm V$ is
called a \emph{coset
of \/$\rm U$ in $\rm V$}.

\begin{lemma}
 \label{t2.1}
Let $(A,B)\in ({\mathbb
C}^{\,\hat{n}\times \hat{n}}_c, {\mathbb
C}^{\,\hat{n}\times \hat{n}}_c)$ and
let $\cal D$ be a pair of $(0,*)$ matrices of the size
$\hat{n}\times \hat{n}$. The
following are
equivalent:
\begin{itemize}
  \item[\rm(i)]
The deformation
$(A,B)+{\cal D}(\vec{\varepsilon})$
defined in \eqref{a2z}
is miniversal.

  \item[\rm(ii)]
The vector space
$({\mathbb
C}^{\,\hat{n}\times \hat{n}}_c, {\mathbb
C}^{\,\hat{n}\times \hat{n}}_c)$
decomposes into the sum
\begin{equation}\label{a4}
({\mathbb C}^{\,{\hat{n}}\times
{\hat{n}}}_c, {\mathbb
C}^{\,\hat{n}\times \hat{n}}_c)=T_{(A,B)} + {\cal
D}({\mathbb C}), \quad T_{(A,B)} \cap {\cal
D}({\mathbb C}) = \{ (A,B) \}.
\end{equation}

  \item[\rm(iii)]
Each coset of
$T_{(A,B)}$
in $({\mathbb
C}^{\,\hat{n}\times \hat{n}}_c, {\mathbb
C}^{\,\hat{n}\times \hat{n}}_c)$
contains exactly one
matrix pair of the form
${\cal D}$.
\end{itemize}
\end{lemma}

\begin{proof}
Define the action of
the group
$GL_{\hat{n}}(\mathbb C)$ of
nonsingular ${\hat{n}}$-by-${\hat{n}}$
matrices on the space
$({\mathbb
C}^{\,\hat{n}\times \hat{n}}_c,{\mathbb
C}^{\,\hat{n}\times \hat{n}}_c)$ by
\[
(A,B)^S=S^T (A,B)S,\qquad
(A,B)\in ({\mathbb
C}^{\,\hat{n}\times \hat{n}}_c,{\mathbb
C}^{\,\hat{n}\times \hat{n}}_c),\quad
S\in GL_{\hat{n}}(\mathbb C).
\]
The orbit $(A,B)^{GL_{\hat{n}}}$
of $(A,B)$ under this
action consists of all pairs
of skew-symmetric
matrices that are
congruent to the pair $(A,B)$.

The space $T_{(A,B)}$
is the tangent space
to the orbit
$(A,B)^{GL_{\hat{n}}}$ at the
point $(A,B)$ (see \eqref{tangatp}). 
Hence ${\cal
D}(\vec
{\varepsilon})$ is
transversal to the
orbit $(A,B)^{GL_{\hat{n}}}$ at
the point $(A,B)$ if
\[
({\mathbb C}^{\,\hat{n}\times \hat{n}}_c,{\mathbb
C}^{\,\hat{n}\times \hat{n}}_c)=T_{(A,B)} +
{\cal D}({\mathbb C})
\]
(see definitions in
\cite[\S\,29]{arn3};
two subspaces of a
vector space are
called
\emph{transversal} if
their sum is equal to
the whole space).

This proves the
equivalence of (i) and
(ii) since a
transversal (of the
minimal dimension) to
the orbit is a
(mini)versal
deformation
\cite[Section
1.6]{arn2}. The
equivalence of (ii)
and (iii) is obvious.
\end{proof}

Due to the versality
of each deformation
$(A,B)+{\cal D}(\vec
{\varepsilon})$ in
which ${\cal D}$
satisfies \eqref{a4}:
there is a
deformation $I(\vec
{\varepsilon})$ of the
identity matrix such
that $(A,B)+{\cal
D}(\vec{\varepsilon})=
I(\vec{\varepsilon})^{T}
{\cal
U}(\vec{\varepsilon})
I(\vec{\varepsilon})$,
where ${\cal U}(\vec
{\varepsilon})$ is
defined in
\eqref{edr}.

Thus, a simplest
miniversal deformation
of $(A,B)\in ({\mathbb
C}^{\,\hat{n}\times \hat{n}}_c, {\mathbb
C}^{\,\hat{n}\times \hat{n}}_c)$ can
be constructed as
follows. Let
$(T_1,\dots,T_r)$ be a
basis of the space
$T_{(A,B)}$,
and let
$(E_1,\dots,E_{{\hat{n}}({\hat{n}}-1)})$ be
the basis of $({\mathbb
C}^{\,\hat{n}\times \hat{n}}_c, {\mathbb
C}^{\,\hat{n}\times \hat{n}}_c)$
in which every $E_k$ is either of the form 
$(E_{ij},0)$ or $(0,E_{i'j'})$. Removing
from the sequence
$(T_1,\dots, T_r,E_1,\dots,E_{{\hat{n}}({\hat{n}}-1)})$
every pair of matrices that is a
linear combination of
the preceding
matrices, we obtain a
new basis $(T_1,\dots, T_r,
E_{i_1},\dots,E_{i_k})$
of the space $({\mathbb
C}^{\,\hat{n}\times \hat{n}}_c, {\mathbb
C}^{\,\hat{n}\times \hat{n}}_c)$. By
Lemma \ref{t2.1}, the
deformation
\begin{align*}
{\cal
A}(\varepsilon_1,\dots,
\varepsilon_{k_1},\varepsilon'_1, \dots, \varepsilon'_{k_2})&=
(A,B)+\varepsilon_1E_{1}+\dots+\varepsilon_{k_1}E_{i_{k_1}} +\varepsilon'_1E_{i_{k_1+1}}+\dots+\varepsilon'_{k_2}E_{i_{k}}\\
&=
(A,B)+\varepsilon_1(E_{i_1,j_1},0)+\dots+\varepsilon_{k_1}(E_{i_{k_1}j_{k_1}},0) \\ 
&+\varepsilon'_1(0,E_{i_{k_1+1},j_{k_1+1}})+\dots+\varepsilon'_{k_2}(0,E_{i_{k},j_{k}}),
\end{align*}
where $k_1+k_2=k$, is miniversal.

For each pair of skew-symmetric $\hat{m}\times \hat{m}$ 
matrices $(A_1,B_1)$ and each pair of 
skew-symmetric $\hat{n}\times \hat{n}$ matrices
$(A_2,B_2)$, define the vector
spaces
\begin{align}\label{neh}
 V(A_1,B_1)&:=\{ S^T(A_1,B_1)+(A_1,B_1)S, \text{ where } S\in
{\mathbb C}^{{\hat{m}}\times {\hat{m}}} \},\\  
\label{neh1}
\begin{split}
 V((A_1,B_1),(A_2,B_2))&:=\{(
R^T(A_2,B_2)+(A_1,B_1)S,
S^T(A_1,B_1)+(A_2,B_2)R), \\ 
&\text{where } S\in {\mathbb C}^{{\hat{m}}\times {\hat{n}}} \text{ and } R\in
 {\mathbb C}^{{\hat{n}}\times {\hat{m}}} \}.
 \end{split}
\end{align}

\begin{lemma}\label{thekd}
Let
$(A,B)=(A_1,B_1)\oplus\dots\oplus
(A_t, B_t)$ be a
block-diagonal matrix
in which every $(A_i,B_i)$
is $n_i\times n_i$.
Let $\cal D$ be a pair of $(0,*)$ matrices of the size of $(A,B)$.
Partitioning $\cal D$ into
blocks $(D_{ij}, D'_{ij})$
conformably to the
partitioning of $(A,B)$
$($see
\eqref{grsd}$)$. Then
$(A,B)+{\cal D}(E,E')$ is a
simplest miniversal (skew-symmetric) 
deformation of $(A,B)$ under
congruence if and only
if
\begin{itemize}
  \item[\rm(i)]
every coset of\/
$V(A_i,B_i)$ in $({\mathbb
C}^{n_i\times n_i}_c, {\mathbb
C}^{n_i\times n_i}_c)$
contains exactly one
matrix of the form
$(D_{ii},D'_{ii})$, and

  \item[\rm(ii)]
every coset of
$V((A_i,B_i),(A_j,B_j))$ in
$({\mathbb
C}^{n_i\times
n_j}, {\mathbb
C}^{n_i\times
n_j}) \oplus  ({\mathbb
C}^{n_j\times n_i}, {\mathbb
C}^{n_j\times n_i})$
contains exactly two pairs of
matrices $((W_1,W_2),(-W_1^T,-W_2^T))$ in
which $(W_1,W_2)$ is of the
form $(D_{ij},D'_{ij})$ and correspondingly $(-W_1^T,-W_2^T)$ is of the form $(D_{ji},D'_{ji})=(-D_{ij}^T,-D_{ij}^{'T})$.
\end{itemize}
\end{lemma}

\begin{proof}
By Lemma
\ref{t2.1}(iii),
$(A,B)+{\cal D}(\vec
{\varepsilon})$ is a
simplest miniversal
deformation of $(A,B)$ if
and only if for each
$(C,C')\in({\mathbb
C}^{\hat{n}\times \hat{n}}_c,{\mathbb
C}^{\,\hat{n}\times \hat{n}}_c)$ the
coset $(C,C')+T_{(A,B)}$
contains exactly one
$(D,D')$ of the form ${\cal
D}$, that is, 
\begin{equation}\label{kid}
(D,D')=(C,C')+S^T(A,B)+(A,B)S\in{\cal
D}(\mathbb C)\qquad
\text{with
$S\in{\mathbb
C}^{\hat{n}\times \hat{n}}$.}
\end{equation}
Partition $(D,D'),\ (C,C')$, and
$S$ into blocks
conformably to the
partitioning of $(A,B)$. By
\eqref{kid}, for each
$i$ we have
$(D_{ii},D'_{ii})=(C_{ii},C'_{ii})+
S_{ii}^T(A_{i},B_{i})
+(A_{i},B_{i})S_{ii}$, and
for all $i$ and $j$
such that $i<j$ we
have
{\small
\begin{multline}\label{mht}
\left(
\begin{bmatrix}
D_{ii}&D_{ij}
 \\ D_{ji}&D_{jj}
\end{bmatrix},
\begin{bmatrix}
D'_{ii}&D'_{ij}
 \\ D'_{ji}&D'_{jj}
\end{bmatrix}
\right)
=
\left(
\begin{bmatrix}
C_{ii}&C_{ij}
 \\ C_{ji}&C_{jj}
\end{bmatrix},
\begin{bmatrix}
C'_{ii}&C'_{ij}
 \\ C'_{ji}&C'_{jj}
\end{bmatrix}
\right) \\
+ \begin{bmatrix}
S_{ii}^T&S_{ji}^T
 \\ S_{ij}^T&S_{jj}^T
\end{bmatrix}
\left(
\begin{bmatrix}
A_i&0
 \\ 0& A_j
\end{bmatrix},
\begin{bmatrix}
B_i&0
 \\ 0& B_j
\end{bmatrix}
\right)
+
\left(
\begin{bmatrix}
A_i&0
 \\ 0& A_j
\end{bmatrix},
\begin{bmatrix}
B_i&0
 \\ 0& B_j
\end{bmatrix}
\right)
\begin{bmatrix}
S_{ii}&S_{ij}
 \\ S_{ji}&S_{jj}
\end{bmatrix}.
\end{multline}
}
Thus, \eqref{kid} is
equivalent to the
conditions
\begin{multline}\label{djh}
(D_{ii},D'_{ii})=(C_{ii},C'_{ii})
+  S_{ii}^T(A_i,B_i)+(A_i,B_i)S_{ii}\in{\cal
D}_{ii}(\mathbb
C),  1\le i\le t,
\end{multline}
\begin{multline}\label{djhh}
((D_{ij},D'_{ij}),(D_{ji},D'_{ji}))=
((C_{ij},C'_{ij}), (C_{ji},C'_{ji})) \\
+ ((S_{ji}^TA_j+A_iS_{ij},S_{ji}^TB_j+B_iS_{ij}),
(S_{ij}^TA_i+A_jS_{ji},S_{ij}^TB_i+B_jS_{ji})) \\
  \in {\cal
D}_{ij}(\mathbb
C)\oplus {\cal
D}_{ji}(\mathbb C), \quad 1\le i<j\le t. 
\end{multline}
Hence for each
$(C,C')\in ({\mathbb
C}^{\hat{n}\times \hat{n}}_c,{\mathbb
C}^{\hat{n}\times \hat{n}}_c)$ there
exists exactly one
$(D,D')\in{\cal D}(\mathbb
C)$ of the
form \eqref{kid} if
and only if
\begin{itemize}
  \item[(i$'$)]
for each
$(C_{ii},C'_{ii})\in({\mathbb
C}^{n_i\times n_i}_c,{\mathbb
C}^{n_i\times n_i}_c)$
there exists exactly
one $(D_{ii},D'_{ii})\in{\cal
D}_{ii}(\mathbb C)$ of the form
\eqref{djh}, and
  \item[(ii$'$)]
for each $((C_{ij},C'_{ij}),
(C_{ji},C'_{ji}))\in ({\mathbb
C}^{n_i\times
n_j},{\mathbb
C}^{n_i\times
n_j})\oplus ({\mathbb
C}^{n_j\times n_i},{\mathbb
C}^{n_j\times n_i})$
there exists exactly
one
$((D_{ij},D'_{ij}),(D_{ji},D'_{ji}))\in
{\cal D}_{ij}(\mathbb
C)\oplus {\cal
D}_{ji}(\mathbb C)$ of
the form \eqref{djhh}.
\end{itemize}
This proves the lemma.
\end{proof}
{\begin{corollary}\label{the}
\textcolor{black}{
In the notation of
Lemma \ref{thekd},
$(A,B)+{\cal D}(\vec
{\varepsilon})$ is a
miniversal deformation
of $(A,B)$ if and only if
each pair of submatrices of the
form
\begin{equation*}\label{a8}
\left(
\begin{bmatrix}
  A_i+D_{ii}(\vec
{\varepsilon}) &
   D_{ij}(\vec
{\varepsilon})\\
   D_{ji}(\vec
{\varepsilon}) &A_j+
D_{jj}(\vec
{\varepsilon})
\end{bmatrix}
\begin{bmatrix}
  B_i+D'_{ii}(\vec
{\varepsilon}) &
  D'_{ij}(\vec
{\varepsilon})\\
 D'_{ji}(\vec
{\varepsilon}) &B_j+
D'_{jj}(\vec
{\varepsilon})
\end{bmatrix}
\right)\quad \text{with } i<j,
\end{equation*}
is a miniversal
deformation of the pair
$(A_i\oplus A_j,B_i\oplus B_j)$.}
\end{corollary}}

We are ready to prove
Theorem \ref{teo2} now.
Each $\mathcal X_i$ in
\eqref{gto} is of the
form $\mathcal H_n(\lambda), \mathcal K_n$, or
$\mathcal L_n$, and so there
are 9 types of pairs
${\cal D}(\mathcal X_i)$ and
${\cal D}(\mathcal X_i, \mathcal X_j)$
with $i<j$; they are
given in 
\eqref{Hdef}--\eqref{ktlm}.
It suffices to prove
that the pairs
\eqref{Hdef}--\eqref{ktlm}
satisfy the conditions
(i) and (ii) of Lemma
\ref{thekd}.

\subsection{Diagonal blocks of $\cal D$}

Fist we verify that
the diagonal blocks of
$\cal D$ defined in
part (i) of Theorem
\ref{teo2} satisfy the
condition (i) of Lemma
\ref{thekd}.

\subsubsection{Diagonal blocks
${\cal D}(\mathcal H_{n}(\lambda))$ and ${\cal D}(\mathcal K_{n})$}
\label{dhndkn}

We consider the pairs of blocks $\mathcal H_{n}(\lambda)$ and $\mathcal K_n$.

Due to Lemma
\ref{thekd}(i), it
suffices to prove that
each pair of skew-symmetric
$2n$-by-$2n$ matrices
$(A,B)=([A_{ij}]_{i,j=1}^2,[B_{ij}]_{i,j=1}^2)$
can be reduced to
exactly one pair of matrices of
the form \eqref{Hdef}
by adding
{\small
\begin{equation}\label{moh1}
\begin{split}
\Delta (A,B)&=(\Delta A, \Delta B)= \left( \begin{bmatrix}
\Delta A_{11}& \Delta A_{12}
 \\ \Delta A _{21}& \Delta A_{22}
\end{bmatrix},\begin{bmatrix}
\Delta B_{11}& \Delta B_{12}
 \\ \Delta B_{21}& \Delta B_{22}
\end{bmatrix} \right)\\ &=
\begin{bmatrix}
S_{11}^T&S_{21}^T
 \\ S_{12}^T&S_{22}^T
\end{bmatrix}
\bigg( \begin{bmatrix}
0&I_n
 \\ -I_n&0
\end{bmatrix},
\begin{bmatrix}
0&J_n(\lambda)
 \\ -J_n(\lambda)^T&0
\end{bmatrix}
\bigg)
 \\ &+ \bigg( \begin{bmatrix} 0&I_n
 \\ -I_n&0
\end{bmatrix},
\begin{bmatrix}
0&J_n(\lambda)
 \\ -J_n(\lambda)^T&0
\end{bmatrix}  \bigg)
\begin{bmatrix}
S_{11}&S_{12}
 \\ S_{21}&S_{22}
\end{bmatrix}
    \\&=
\bigg( \begin{bmatrix}
S_{21}-S_{21}^T&
S_{11}^T+S_{22}\\
-S_{11}-S_{22}^T&
S_{12}^T-S_{12}
\end{bmatrix}, \\&
\begin{bmatrix}
-S_{21}^TJ_n(\lambda)^T+J_n(\lambda)S_{21}&
S_{11}^TJ_n(\lambda)+J_n(\lambda)S_{22}\\
-S_{22}^TJ_n(\lambda)^T-J_n(\lambda)^TS_{11}&
S_{12}^TJ_n(\lambda)-J_n(\lambda)^TS_{12}
\end{bmatrix} \bigg),
\end{split}
\end{equation}}\noindent in which
$S=[S_{ij}]_{i,j=1}^2$
is an arbitrary
$2n$-by-$2n$ matrix. Due to the skew-symmetry there are three pairs of $n$-by-$n$ blocks in \eqref{moh1} that can be treated independently. For any $X$ we have $$-XJ_n(\lambda)^T+J_n(\lambda)X=
-X(\lambda I+J_n(0))^T+(\lambda I+J_n(0))X=-XJ_n(0)^T+J_n(0)X.$$
Thus, without loss of generality, we can assume that $\lambda=0$.
Therefore the deformation of $\mathcal K_{n}$ is equal to the
deformation of $\mathcal H_{n}(\lambda)$ up to the permutation of
matrices.

First we consider the pair of blocks
$\Delta (A_{11},B_{11})=(S_{21}-S_{21}^T,-S_{21}^TJ_n(0)^T+J_n(0)S_{21})$
in which $S_{21}$ is an arbitrary $n$-by-$n$ matrix. Obviously, by 
adding $\Delta A_{11}=S_{21}-S_{21}^T$ we reduce $A_{11}$ to zero. To
preserve $A_{11},$ we must hereafter take $S_{21}$ such that $S_{21}-S_{21}^T=0,$ i.e., $S_{21}$ is symmetric.
We reduce $B_{11}$ by adding
$\Delta B_{11} = -S_{21}^TJ_n(0)^T+J_n(0)S_{21}$,
\begin{equation} \label{qwert11}
\footnotesize{
\begin{split}
&\Delta B_{11}=\\&=-\begin{bmatrix}
s_{11}&s_{12}&s_{13}&\ldots & s_{1n}\\
s_{12}&s_{22}&s_{23}&\ldots & s_{2n}\\
s_{13}&s_{23}&s_{33}&\ldots & s_{3n}\\
\vdots&\vdots&\vdots&\ddots& \vdots\\
s_{1n}&s_{2n}&s_{3n}&\ldots & s_{nn}\\
\end{bmatrix}
\begin{bmatrix}
0&&&0\\
1&0&&\\
&\ddots&\ddots & \\
0&&1 & 0\\
\end{bmatrix}
 +
\begin{bmatrix}
0&1&&0\\
&0&\ddots & \\
&&\ddots & 1\\
0&&& 0\\
\end{bmatrix}
\begin{bmatrix}
s_{11}&s_{12}&s_{13}&\ldots & s_{1n}\\
s_{12}&s_{22}&s_{23}&\ldots & s_{2n}\\
s_{13}&s_{23}&s_{33}&\ldots & s_{3n}\\
\vdots&\vdots&\vdots&\ddots& \vdots\\
s_{1n}&s_{2n}&s_{3n}&\ldots & s_{nn}\\
\end{bmatrix}
    \\&=
\begin{bmatrix}
0&s_{22}-s_{13}&s_{23}-s_{14}&\ldots & s_{2n}\\
-s_{22}+s_{13}&0&s_{33}-s_{24}&\ldots & s_{3n}\\
-s_{23}+s_{14}&-s_{33}+s_{24}&0&\ldots & s_{4n}\\
\vdots&\vdots&\vdots&\ddots& \vdots\\
-s_{2n}&-s_{3n}&-s_{4n}&\ldots & 0\\
\end{bmatrix}.
\end{split}
}
\end{equation}
We reduce $B_{11}$ anti-diagonal-wise and since $B_{11}$ is skew-symmetric, we just need to reduce the upper triangular part of $B_{11}$ and the lower triangular part will be reduced automatically. Let $b =(b_1, \ldots, b_{t-1})$ denote the elements of the upper half of the $k$-th anti-diagonal (counting from the top left corner) of $B_{11}$. Each of the first $(n-1)$ upper halfs of the anti-diagonals of $\Delta B_{11}$ is of the form 
$$s = 
\begin{cases}
(s_{2k}-s_{1,k+1}, s_{3,k-1}-s_{2k}, \ldots , s_{tt}-s_{t-1,t+1}), \ \ \text{if }
$k$ \text{ is even, } t = \frac{k+2}{2};  \\
(s_{2k}-s_{1,k+1}, s_{3,k-1}-s_{2k}, \ldots , s_{t,t+1}-s_{t-1,t+2}), \ \ \text{if } $k$ \text{ is odd, } t = \frac{k+1}{2},  
\end{cases}
$$ 
where $k=2, 3, \dots , n-1,$ (the first anti-diagonal is zero). 
Choosing the parameters $s_{ij}$ we want to make $s$ equal to $b$, i.e. we want to solve the system of linear equations
\begin{equation} \label{4sdf}
\small{
\left[
 \begin{matrix}
-1&1&&&0\\
&-1&1&&\\
&&\ddots&\ddots&\\
0&&&-1&1
\end{matrix}
\right]
\left[
 \begin{matrix}
s_{1,k+1} \\
s_{2k} \\
\vdots \\
s_{tt}
\end{matrix}
\right]
=
\left[
 \begin{matrix}
b_1 \\
b_2 \\
\vdots \\
b_{t-1}
\end{matrix}
\right],
}
\end{equation}
where $k$ is even (and the analogous system for $k$ being odd).
The system \eqref{4sdf} has a solution. 
Therefore, we can reduce each of the first $(n-1)$ anti-diagonals of $B_{11}$ to zero, by adding the corresponding anti-diagonals of $\Delta B_{11}$.

For each $k$-th upper parts 
of the last $n$ anti-diagonals we have the following systems of equations 
\begin{equation} \label{systsov}
\small{
\left[
 \begin{matrix}
 1&&&&0\\
-1&1&&&\\
&-1&1&&\\
&&\ddots&\ddots&\\
0&&&-1&1 \\
\end{matrix}
\right]
\left[
 \begin{matrix}
s_{2-n+k,n} \\
s_{3-n+k,n-1} \\
\vdots \\
s_{t'-1,t'+1}\\
s_{t't'}
\end{matrix}
\right]
=
\left[
 \begin{matrix}
b_1 \\
b_2 \\
\vdots \\
b_{t'-2}\\
b_{t'-1}
\end{matrix}
\right],
}
\end{equation}
where $k=n, n+1, \dots , 2n-2,$ (the last anti-diagonal is zero) and $t'=t-k+n$ and $t$ is defined as above. 
The system \eqref{systsov} has a solution. 
Therefore we can reduce the last $n$ anti-diagonals of $B_{11}$ to zero.
Altogether, we reduce $B_{11}$ to zero matrix by adding $\Delta B_{11}$.

The possibility of reducing $(A_{22},B_{22})$ to zero by adding $\Delta (A_{22},B_{22})=(S_{12}^T-S_{12},S_{12}^TJ_n(0)-J_n(0)^TS_{12})$ follows directly from the reduction of the blocks $(A_{11},B_{11})$.
We have $0=B_{11}-S_{21}^TJ_n(0)^T+J_n(0)S_{21}$ where $B_{11}$ is a skew-symmetric matrix. Multiplying
this equality by the $n$-by-$n$ flip matrix
\begin{equation} \label{zzz}
Z:= \begin{bmatrix}
 0&&1
 \\ &\udots&\\
 1&&0
 \end{bmatrix}
 \end{equation}
 from both sides and using that $Z^2=I$ and
 $ZJ_n(0)^TZ=J_n(0)$ we get
 $$0=ZB_{11}Z-ZS_{21}^TZJ_n(0)+J_n(0)^TZS_{21}Z.$$
This ensures that the pair of blocks $(A_{22},B_{22})$ can be set to zero
since $ZB_{11}Z$ and $ZS_{21}Z$ are arbitrary skew-symmetric and symmetric matrices,
respectively.

To the pair of blocks $(A_{21},B_{21})$ we can add
$\Delta (A_{21},B_{21})=(S_{11}^T+S_{22},S_{11}^TJ_n(0)+J_n(0)S_{22})$.
Adding $S_{11}^T+S_{22}$ we reduce $A_{21}$ to zero.
To preserve $A_{21}$, we must hereafter take $S_{11}$ and $S_{22}$
such that $S_{11}^T=-S_{22}$.
Thus we add $\Delta B_{21}=-S_{22}J_n(0)+J_n(0)S_{22}$, with any matrix $S_{22}$,
\begin{equation} \label{45form}
\footnotesize{
\begin{split}
&\Delta B_{21}=\\&=-\begin{bmatrix}
s_{11}&s_{12}&s_{13}&\ldots & s_{1n}\\
s_{21}&s_{22}&s_{23}&\ldots & s_{2n}\\
s_{31}&s_{32}&s_{33}&\ldots & s_{3n}\\
\vdots&\vdots&\vdots&\ddots& \vdots\\
s_{n1}&s_{n2}&s_{n3}&\ldots & s_{nn}\\
\end{bmatrix}
\begin{bmatrix}
0&1&&\\
&0&\ddots & \\
&&\ddots & 1\\
&&& 0\\
\end{bmatrix}
 +
\begin{bmatrix}
0&1&&\\
&0&\ddots & \\
&&\ddots & 1\\
&&& 0\\
\end{bmatrix}
\begin{bmatrix}
s_{11}&s_{12}&s_{13}&\ldots & s_{1n}\\
s_{21}&s_{22}&s_{23}&\ldots & s_{2n}\\
s_{31}&s_{32}&s_{33}&\ldots & s_{3n}\\
\vdots&\vdots&\vdots&\ddots& \vdots\\
s_{n1}&s_{n2}&s_{n3}&\ldots & s_{nn}\\
\end{bmatrix}
    \\&=
\begin{bmatrix}
s_{21}&s_{22}-s_{11}&s_{23}-s_{12}&\ldots & s_{2n}-s_{1,n-1}\\
s_{31}&s_{32}-s_{21}&s_{33}-s_{22}&\ldots & s_{3n}-s_{2,n-1}\\
s_{41}&s_{42}-s_{31}&s_{43}-s_{32}&\ldots & s_{4n}-s_{3,n-1}\\
\vdots&\vdots&\vdots&\ddots& \vdots\\
0&-s_{n1}&-s_{n2}&\ldots & -s_{n,n-1}\\
\end{bmatrix}.
\end{split}
}
\end{equation}
We examine each diagonal of $\Delta B_{21}$ independently since 
each diagonal has unique variables. For each of the first $n$ diagonals
(starting from the bottom left corner) we
have the following system of equations 
\begin{equation}\label{44sdf}
\small{
\left[
 \begin{matrix}
1&&& \\
-1&1&&  \\
&\ddots&\ddots&\\
&&-1&1 \\
&&& -1 \\
\end{matrix}
\right]
\left[
 \begin{matrix}
s_{n+2-k, 1} \\
\vdots \\
s_{n, k-1}
\end{matrix}
\right]
=
\left[
 \begin{matrix}
b_1 \\
\vdots \\
b_{k}
\end{matrix}
\right].
}
\end{equation}
The matrix of this system has $k-1$ columns and $k$ (since the first diagonal is zero $k=2,\dots , n$) rows and
its rank is equal to $k-1$ but the rank of the full matrix of
the system is $k$; by the Kronecker-Capelli theorem \cite{k} the system \eqref{44sdf} does not have
a solution. Nevertheless, if we turn down the first or the last equation of the system
(i.e. we do not set the first or the last element of the corresponding diagonal of $B_{21}$ to zero), then \eqref{44sdf} will have a solution.

For the last $(n-1)$ diagonals we have a system  of equations like \eqref{4sdf}, which has a solution. 
Therefore we can set each element
of the matrix $B_{21}$ to zero except the elements either in the first column or the last row.

The blocks $\Delta (A_{12},B_{12})=(-S_{11}-S_{22}^T,-S_{22}^TJ_n(0)^T-J_n(0)^TS_{11})$ are equal to $\Delta (A_{21},B_{21})$ up to the transposition and sign.

Altogether, we obtain
$${\cal D}(\mathcal H_{n}(\lambda))=
\left( 0,
\begin{bmatrix}
0&0^{\swarrow}\\
0^{\nearrow}&0
\end{bmatrix} \right)
\quad \text{and} \quad {\cal D}(\mathcal K_{n}) = \left(
\begin{bmatrix}
0&0^{\swarrow}\\
0^{\nearrow}&0
\end{bmatrix},0 \right).$$

\subsubsection{Diagonal blocks ${\cal D}(L_{n})$}
\label{dln}

Using Lemma \ref{thekd}(i), like in Section \ref{dhndkn},
we prove that each pair $(A,B)=([A_{ij}]_{i,j=1}^2,[B_{ij}]_{i,j=1}^2)$ of skew-symmetric $(2n+1)$-by-$(2n+1)$ matrices can be set to zero by adding
\begin{equation}\label{moh}
\small{
\begin{split}
&\Delta (A,B)=(\Delta A, \Delta B)= \left( \begin{bmatrix}
\Delta A_{11}& \Delta A_{12}
 \\ \Delta A_{21}&\Delta A_{22}
\end{bmatrix}, \begin{bmatrix}
\Delta B_{11}& \Delta B_{12}
 \\ \Delta B_{21}&\Delta B_{22}
\end{bmatrix} \right) \\ 
&= \begin{bmatrix}
S_{11}^T&S_{21}^T
 \\ S_{12}^T&S_{22}^T
\end{bmatrix}
\bigg( \begin{bmatrix}
0&F_n \\
-F_n^T&0
\end{bmatrix},
\begin{bmatrix}
0&G_n \\
-G_n^T&0
\end{bmatrix}
\bigg) + \left( \begin{bmatrix}
0&F_n \\
-F_n^T&0
\end{bmatrix},
\begin{bmatrix}
0&G_n \\
-G_n^T&0
\end{bmatrix}  \right)
\begin{bmatrix}
S_{11}&S_{12}
 \\ S_{21}&S_{22}
\end{bmatrix}
    \\&=
\bigg( \begin{bmatrix}
-S_{21}^TF_n^T+F_nS_{21}&
S_{11}^TF_n+F_nS_{22}\\
-S_{22}^TF_n^T-F_n^TS_{11}&
S_{12}^TF_n-F_n^TS_{12}
\end{bmatrix}, 
\begin{bmatrix}
-S_{21}^TG_n^T+G_nS_{21}&
S_{11}^TG_n+G_nS_{22}\\
-S_{22}^TG_n^T-G_n^TS_{11}&
S_{12}^TG_n-G_n^TS_{12}
\end{bmatrix} \bigg),
\end{split}
}
\end{equation}
where
    $S=[S_{ij}]_{i,j=1}^2$
is an arbitrary matrix. Each pair of blocks $(A_{ij},B_{ij}), i,j=1,2,$ of $(A,B)$ is changed
independently.

We add
$\Delta (A_{11},B_{11})=(-S_{21}^TF_n^T+F_nS_{21},-S_{21}^TG_n^T+G_nS_{21})$
in which $S_{21}$ is an arbitrary $(n+1)$-by-$n$
matrix to the pair of blocks $(A_{11},B_{11})$. Obviously, by adding $-S_{21}^TF_n^T+F_nS_{21}$
we reduce $A_{11}$ to zero. To preserve $A_{11}$, we must hereafter take 
$S_{21}$  such that $F_nS_{21}=S_{21}^TF_n^T$.
Thus $S_{21}$ without the last row is $n\times n$ and symmetric: 
{\small
$$S_{21}=\begin{bmatrix}
s_{11}&s_{12}&s_{13}&\ldots & s_{1n}\\
s_{12}&s_{22}&s_{23}&\ldots & s_{2n}\\
s_{13}&s_{23}&s_{33}&\ldots & s_{3n}\\
\vdots&\vdots&\vdots&\ddots & \vdots\\
s_{1n}&s_{2n}&s_{3n}&\ldots & s_{nn}\\
s_{1,n+1}&s_{2,n+1}&s_{3,n+1}&\ldots & s_{n,n+1}\\
\end{bmatrix}.$$
}
Now we reduce $B_{11}$ by adding
{\small
\begin{multline}
\Delta B_{11}=-\begin{bmatrix}
s_{11}&s_{12}&s_{13}&\ldots & s_{1n}& s_{1,n+1}\\
s_{12}&s_{22}&s_{23}&\ldots & s_{2n}& s_{2,n+1}\\
s_{13}&s_{23}&s_{33}&\ldots & s_{3n}& s_{3,n+1}\\
\vdots&\vdots&\vdots&\ddots & \vdots & \vdots\\
s_{1n}&s_{2n}&s_{3n}&\ldots & s_{nn}& s_{n,n+1}\\
\end{bmatrix}
\begin{bmatrix}
0& &0\\
1&\ddots &\\
&\ddots & 0\\
0& & 1\\
\end{bmatrix}
 \\ +
\begin{bmatrix}
0&1&&0\\
&\ddots&\ddots&\\
0&& 0&1\\
\end{bmatrix}
\begin{bmatrix}
s_{11}&s_{12}&s_{13}&\ldots & s_{1n}\\
s_{12}&s_{22}&s_{23}&\ldots & s_{2n}\\
s_{13}&s_{23}&s_{33}&\ldots & s_{3n}\\
\vdots&\vdots&\vdots&\ddots & \vdots\\
s_{1n}&s_{2n}&s_{3n}&\ldots & s_{nn}\\
s_{1,n+1}&s_{2,n+1}&s_{3,n+1}&\ldots & s_{n,n+1}\\
\end{bmatrix}
    \\=
\begin{cases}
 \begin{matrix}
-s_{i,j+1}+s_{i+1,j}&  \text{if} \ \ \  i<j, \\
s_{i,j+1}-s_{i+1,j}&  \text{if} \ \ \ i>j, \\
0& \text{if} \ \ \  i=j,
\end{matrix} \ \ \  
\end{cases}
\end{multline}}\noindent where $i,j=1, \ldots , n$.
The upper part of each anti-diagonal of $\Delta B_{11}$ has unique variables. Thus we
reduce each anti-diagonal of $B_{11}$ independently. We
have a system of equations \eqref{4sdf} for
the upper part of each anti-diagonal,
which has a solution.
It follows that we can reduce every anti-diagonal of $B_{11}$ to zero.
Hence we can reduce $(A_{11},B_{11})$ to zero by adding $\Delta (A_{11},B_{11})$.

To the pair of blocks $(A_{12},B_{12})$ we can add
$\Delta (A_{12},B_{12})=(S_{11}^TF_n+F_nS_{22},S_{11}^TG_n+G_nS_{22})$
in which $S_{11}$ and $S_{22}$ are arbitrary matrices of corresponding size. Adding
$S_{11}^TF_n+F_nS_{22}$, we reduce $A_{12}$ to zero.
To preserve $A_{12}$, we must hereafter take $S_{11}$ and $S_{22}$  such that $F_nS_{22}=-S_{11}^TF_n$. This means that
{\small
$$S_{22}=\begin{bmatrix}
&&& 0\\
&-S^T_{11}&& 0\\
&&& \vdots\\
&&& 0\\
y_{1}&y_{2}&\ldots & y_{n+1}\\
\end{bmatrix}.$$
}
Therefore we reduce $B_{12}$ by adding
{\small
\begin{multline}
\Delta B_{12}=S_{11}^TG_n+G_nS_{22}=
\begin{bmatrix}
s_{11}&s_{12}&s_{13}&\ldots & s_{1n}\\
s_{21}&s_{22}&s_{23}&\ldots & s_{2n}\\
s_{31}&s_{32}&s_{33}&\ldots & s_{3n}\\
\vdots&\vdots&\vdots&\ddots & \vdots\\
s_{n1}&s_{n2}&s_{n3}&\ldots & s_{nn}\\
\end{bmatrix}
\begin{bmatrix}
0&1&&0\\
&\ddots&\ddots&\\
0&& 0&1\\
\end{bmatrix}
 \\ +
\begin{bmatrix}
0&1&&0\\
&\ddots&\ddots&\\
0&& 0&1\\
\end{bmatrix}
\begin{bmatrix}
-s_{11}&-s_{12}&\ldots & -s_{1n}& 0\\
-s_{21}&-s_{22}&\ldots & -s_{2n}& 0\\
\vdots&\vdots&\ddots & \vdots & \vdots\\
-s_{n1}&-s_{n2}&\ldots & -s_{nn}&0\\
y_{1}&y_{2}&\ldots &y_{n} & y_{n+1}\\
\end{bmatrix}
    \\=-
\begin{bmatrix}
s_{21}& -s_{11}+s_{22}& -s_{12}+s_{23}&  \ldots & -s_{1,n-1}+s_{2n}& -s_{1n} \\
s_{31}& -s_{21}+s_{32}& -s_{22}+s_{33}&  \ldots & -s_{2,n-1}+s_{3n}& -s_{2n} \\
\hdotsfor{6}\\
s_{n1}& -s_{n-1,1}+s_{n2}& -s_{n-1,2}+s_{n3}& \ldots & -s_{n,n-1}+s_{nn}& -s_{n-1n} \\
-y_1& -s_{n1}-y_2 & -s_{n2}-y_3 &  \ldots & -s_{n,n-1}-y_n & -s_{nn}-y_{n+1} \\
\end{bmatrix}.
\end{multline}
}
It is easily seen that we can set $B_{12}$ to zero by adding $\Delta B_{12}$
(diagonal-wise).

The pair of blocks $\Delta (A_{21},B_{21})=(-S_{22}^TF_n^T-F_n^TS_{11},-S_{22}^TG_n^T-G_n^TS_{11})$ is
analogous to $\Delta (A_{12},B_{12})$ up to transposition and sign.

To the pair of blocks $(A_{22},B_{22})$ we add
$\Delta (A_{22},B_{22})=(S_{12}^TF_n-F_n^TS_{12},S_{12}^TG_n-G_n^TS_{12})$
in which $S_{12}$ is an arbitrary $n$-by-$(n+1)$
matrix. Obviously, by adding $S_{12}^TF_n-F_n^TS_{12}$,
we reduce $A_{22}$ to
zero. To preserve $A_{22}$, we must hereafter take $S_{12}$ such that $S_{12}^TF_n=F_n^TS_{12}$.
Thus
{\small
$$S_{12}=\begin{bmatrix}
s_{11}&s_{12}&s_{13}&\ldots & s_{1n}& 0\\
s_{12}&s_{22}&s_{23}&\ldots & s_{2n}& 0\\
s_{13}&s_{23}&s_{33}&\ldots & s_{3n}& 0\\
\vdots&\vdots&\vdots&\ddots & \vdots & \vdots\\
s_{1n}&s_{2n}&s_{3n}&\ldots & s_{nn}& 0\\
\end{bmatrix}.$$
}
The matrix $S_{12}$ without the last column is $n \times n$ and symmetric.
Now we reduce $B_{22}$ by adding
\begin{equation*}
\footnotesize{
\begin{split}
&\Delta B_{22}= \\ &=
\begin{bmatrix}
s_{11}&s_{12}&s_{13}&\ldots & s_{1n}\\
s_{12}&s_{22}&s_{23}&\ldots & s_{2n}\\
s_{13}&s_{23}&s_{33}&\ldots & s_{3n}\\
\vdots&\vdots&\vdots&\ddots & \vdots\\
s_{1n}&s_{2n}&s_{3n}&\ldots & s_{nn}\\
0&0&0&\ldots & 0\\
\end{bmatrix}
\begin{bmatrix}
0&1&&0\\
&\ddots&\ddots&\\
0&& 0&1\\
\end{bmatrix}
-\begin{bmatrix}
0&&0\\
1&\ddots &\\
&\ddots & 0\\
0& & 1\\
\end{bmatrix}
\begin{bmatrix}
s_{11}&s_{12}&s_{13}&\ldots & s_{1n}& 0\\
s_{12}&s_{22}&s_{23}&\ldots & s_{2n}& 0\\
s_{13}&s_{23}&s_{33}&\ldots & s_{3n}& 0\\
\vdots&\vdots&\vdots&\ddots & \vdots & \vdots\\
s_{1n}&s_{2n}&s_{3n}&\ldots & s_{nn}& 0\\
\end{bmatrix}
    \\&=
\begin{bmatrix}
0& s_{11}& s_{12}& s_{13}& \ldots & s_{1,n-1}& s_{1n} \\
-s_{11}& 0& s_{22}-s_{13}& s_{23}-s_{14}& \ldots& s_{2,n-1}-s_{1n}& s_{2n} \\
-s_{12}& s_{13}-s_{22}& 0& s_{33}-s_{24}& \ldots& s_{3,n-1}-s_{2n}& s_{3n} \\
-s_{13}& s_{14}-s_{23}& s_{24}-s_{33}& 0& \ldots& s_{4,n-1}-s_{3n}& s_{4n} \\
\hdotsfor{7}\\
-s_{1,n-1}& s_{1n}-s_{2,n-1}& s_{2n}-s_{3,n-1}& s_{3n}-s_{4,n-1}& \ldots& 0 & s_{nn} \\
-s_{1n}& -s_{2n}& -s_{3n}& -s_{4n} & \ldots& -s_{nn}& 0
\end{bmatrix}.
\end{split}
}
\end{equation*}
We have a system of equations of type \eqref{systsov}
which has a solution for the upper part of each anti-diagonal.
It follows that we can reduce every anti-diagonal of $B_{22}$ to zero. Hence we can reduce $(A_{22},B_{22})$ to zero by adding $\Delta (A_{22},B_{22})$.

Summing up the analysis for all pairs of blocks, we get ${\cal D}(L_{n})=0$.

\subsection{Off-diagonal
blocks of
$\cal D$ that
correspond to summands
of
$(A,B)_{\text{can}}$
of the same type}

Now we verify the
condition (ii) of
Lemma \ref{thekd} for
off-diagonal blocks of
$\cal D$ defined in
Theorem
\ref{teo2}(ii); the
diagonal blocks of
their horizontal and
vertical strips
contain summands of
$(A,B)_{\text{can}}$ of the same type.

\subsubsection{Pairs of
blocks ${\cal D}(\mathcal H_n(\lambda),\, \mathcal H_m(\mu))$ and
${\cal D}(\mathcal K_n,\mathcal K_m)$}
\label{sub4}

Due to Lemma
\ref{thekd}(ii), it
suffices to prove that
each group of four matrices $((A,B),(-A^T,-B^T))$
can be reduced to
exactly one group of
the form \eqref{lsiu1}
by adding
\[
(R^T\mathcal H_m(\mu)
+\mathcal H_n(\lambda)S, S^T
\mathcal H_n(\lambda)+ \mathcal H_m(\mu))
R),\quad S\in
 {\mathbb C}^{2n\times 2m}, R\in
 {\mathbb C}^{2m\times 2n}.
\]

Obviously, if we reduce the first
pair of matrices, the second pair will be reduced automatically.
So we reduce a pair $(A,B)$ of $2n$-by-$2m$ matrices by
adding
{\small
\begin{multline*}
\Delta (A,B)=R^T\mathcal H_m(\mu)
+ \mathcal H_n(\lambda)S = \\ \left(R^T
\begin{bmatrix}
0&I_m\\
-I_m&0\\
\end{bmatrix}
+
\begin{bmatrix}
0&I_n\\
-I_n&0\\
\end{bmatrix}
S,R^T
\begin{bmatrix}
0&J_m(\mu)\\
-J_m(\mu)^T&0\\
\end{bmatrix}
+
\begin{bmatrix}
0&J_n(\lambda)\\
-J_n(\lambda)^T&0\\
\end{bmatrix}
S\right) .
\end{multline*} }\noindent It is clear that we can reduce $A$ to zero. To preserve
$A$, we must hereafter choose $R=[R_{ij}]_{i,j=1}^2$ and $S=[S_{ij}]_{i,j=1}^2$ such that
{\small
\[
R^T
\begin{bmatrix}
0&I_m\\
-I_m&0\\
\end{bmatrix}
+
\begin{bmatrix}
0&I_n\\
-I_n&0\\
\end{bmatrix}
S=0, \text{ \normalsize{or equivalently} }
\begin{bmatrix}
S_{11}&S_{12}\\
S_{21}&S_{22}\\
\end{bmatrix}=
\begin{bmatrix}
-R^T_{22}&R^T_{12}\\
R^T_{21}&-R^T_{11}\\
\end{bmatrix}.
\]
}
Now $B:=[B_{ij}]_{i,j=1}^2
$ is reduced by adding
{\small
\begin{align*}
\Delta B &:=
\begin{bmatrix}
R^T_{11}&R^T_{21}\\
R^T_{12}&R^T_{22}\\
\end{bmatrix}
\begin{bmatrix}
0&J_m(\mu)\\
-J_m(\mu)^T&0\\
\end{bmatrix}
+
\begin{bmatrix}
0&J_n(\lambda)\\
-J_n(\lambda)^T&0\\
\end{bmatrix}
\begin{bmatrix}
-R^T_{22}&R^T_{12}\\
R^T_{21}&-R^T_{11}\\
\end{bmatrix} \\
&= \begin{bmatrix}
-R^T_{21}J_m(\mu)^T+J_n(\lambda)R_{21}^T&R^T_{11}J_m(\mu)-J_n(\lambda)R_{11}^T\\
-R^T_{22}J_m(\mu)^T+J_n(\lambda)^TR_{22}^T&R^T_{12}J_m(\mu)-J_n(\lambda)^TR_{12}^T\\
\end{bmatrix}.
\end{align*}
}
Therefore $B_{11}$ is reduced by adding
{\small
\begin{align*}
\Delta B_{11}&=-R^T_{21}J_m(\mu)^T+J_n(\lambda)R_{21}^T \\ &=
\begin{cases}
(\lambda-\mu)r_{ij}+r_{i+1,j}-r_{i,j+1}&  \text{if } 1 \leq i \leq (n-1), 1 \leq j \leq (m-1), \\
(\lambda-\mu)r_{ij}+r_{i+1,j}&  \text{if } 1 \leq i \leq (n-1), j=m, \\
(\lambda-\mu)r_{ij}-r_{i,j+1}&  \text{if } 1 \leq j \leq (m-1),  i=n, \\
(\lambda-\mu)r_{ij}& \text{if } i=n, j=m. 
\end{cases}
\end{align*}}\noindent
We have a system of $nm$ equations which has a solution if $\lambda \neq \mu.$
Thus for $\lambda \neq \mu$ we can set $B_{11}$  to zero by adding
$\Delta B_{11}$.

Now we consider $\lambda = \mu,$ i.e.
{\small
\begin{align*}
\Delta B_{11}&=-R^T_{21}J_m(\lambda)^T+J_n(\lambda)R_{21}^T \\
&=\begin{bmatrix}
r_{21}-r_{12}&r_{22}-r_{13}&r_{23}-r_{14}&  \ldots & r_{2,m-1}-r_{1m}& r_{2m}\\
r_{31}-r_{22}&r_{32}-r_{23}&r_{33}-r_{24}&  \ldots & r_{3,m-1}-r_{2m}& r_{3m}\\
r_{41}-r_{32}&r_{42}-r_{33}&r_{43}-r_{34}&  \ldots & r_{4,m-1}-r_{3m}& r_{4m}\\
\hdotsfor{6}\\
r_{n1}-r_{n-1,2}&r_{n2}-r_{n-1,3}&r_{n3}-r_{n-1,4}&  \ldots & r_{n,m-1}-r_{n-1,m}& r_{nm}\\
-r_{n2}&-r_{n3}&-r_{n4}&  \ldots & -r_{nm}& 0\\
\end{bmatrix}.
\end{align*}
}
Like for the system \eqref{45form}, $B_{11}$ can be reduced to $0^{\searrow}$ by adding $\Delta B_{11}$. 

To find the solutions for the other cases we need to multiply the answer for the block $B_{11}$ by $\pm Z$:
{\small
\begin{equation*}
\begin{split}
B_{12}&-R^T_{11}J_m(\mu)-J_n(\lambda)R_{11}^T \\
&=B_{11}Z+R^T_{21}ZZJ_m(\mu)^TZ-J_n(\lambda)R_{21}^TZ= \begin{cases} 0Z=0 &\text{if } \lambda \neq \mu, \\
0^{\searrow}Z=0^{\swarrow} &\text{if } \lambda = \mu, \end{cases} \\
B_{21}&-R^T_{22}J_m(\mu)^T+J_n(\lambda)^TR_{22}^T \\
&=ZB_{11}+ZR^T_{21}J_m(\mu)^T-ZJ_n(\lambda)ZZR_{21}^T= \begin{cases} Z0=0 &\text{if } \lambda \neq \mu, \\
Z0^{\searrow}=0^{\nearrow}&\text{if } \lambda = \mu, \end{cases}\\
B_{22}&-R^T_{12}J_m(\mu)-J_n(\lambda)^TR_{12}^T \\
&=ZB_{11}Z+ZR^T_{21}ZZJ_m(\mu)^TZ-ZJ_n(\lambda)ZZR_{21}^TZ= \begin{cases} Z0Z=0 &\text{if } \lambda \neq \mu, \\
Z0^{\searrow}Z=0^{\nwarrow} &\text{if } \lambda = \mu. \end{cases}
\end{split}
\end{equation*}
}

Summing up the derivations for all blocks, we get
that ${\cal D}(\mathcal H_n(\lambda),\, \mathcal H_m(\mu))$ is equal to \eqref{lsiu1} and, respectively,
${\cal D}(\mathcal K_n,\, \mathcal K_m)$ is equal to \eqref{lsiu2}.

\subsubsection{Pairs of
blocks ${\cal D}(\mathcal L_n, \mathcal L_m)$ }
\label{sub42}

Due to Lemma
\ref{thekd}(ii), it
suffices to prove that
each group of four matrices $((A,B),(-A^T,-B^T))$
can be reduced to
exactly one group of
the form \eqref{lsiu3}
by adding
\[
(R^T\mathcal L_m
+\mathcal L_nS, S^T
\mathcal L_n+ \mathcal L_m R),\quad S\in
 {\mathbb C}^{2n+1\times 2m+1},\ R\in
 {\mathbb C}^{2m+1\times 2n+1}.
\]

It is enough to reduce only the first
pair of matrices, i.e. $(A,B)$. We reduce it by adding 
{\small
\begin{multline*}
\Delta (A,B)=R^T\mathcal L_m
+\mathcal L_n S \\ = \left(R^T
\begin{bmatrix}
0&F_m\\
-F^T_m&0\\
\end{bmatrix}
+
\begin{bmatrix}
0&F_n\\
-F^T_n&0\\
\end{bmatrix}
S,R^T
\begin{bmatrix}
0&G_m\\
-G^T_m&0\\
\end{bmatrix}
+
\begin{bmatrix}
0&G_n\\
-G^T_n&0\\
\end{bmatrix}
S \right) .
\end{multline*}}\noindent
It is easily seen that we can set $A$ to zero. To preserve
$A,$ we must hereafter take $R=[R_{ij}]_{i,j=1}^2$ and $S=[S_{ij}]_{i,j=1}^2$ such that
{\small
\[
\begin{bmatrix}
R^T_{11}&R^T_{21}\\
R^T_{12}&R^T_{22}\\
\end{bmatrix}
\begin{bmatrix}
0&F_m\\
-F^T_m&0\\
\end{bmatrix}
+
\begin{bmatrix}
0&F_n\\
-F^T_n&0\\
\end{bmatrix}
\begin{bmatrix}
S_{11}&S_{12}\\
S_{21}&S_{22}\\
\end{bmatrix}=0,
\]
}
or equivalently 
{\small
\begin{equation} \label{spiv6}
\begin{bmatrix}
-R^T_{21}F_m^T&R^T_{11}F_m\\
-R^T_{22}F_m^T&R^T_{12}F_m\\
\end{bmatrix}=
\begin{bmatrix}
-F_nS_{21}&-F_nS_{22}\\
F_n^TS_{11}&F_n^TS_{12}\\
\end{bmatrix}.
\end{equation}
}
$B:=[B_{ij}]_{i,j=1}^2$ is reduced by adding
{\small
\begin{align*}
\Delta B&:= \begin{bmatrix}
\Delta B_{11}&\Delta B_{12}\\
\Delta B_{21}&\Delta B_{22}\\
\end{bmatrix}  = \begin{bmatrix}
R^T_{11}&R^T_{21}\\
R^T_{12}&R^T_{22}\\
\end{bmatrix}
\begin{bmatrix}
0&G_m\\
-G^T_m&0\\
\end{bmatrix}
+
\begin{bmatrix}
0&G_n\\
-G^T_n&0\\
\end{bmatrix}
\begin{bmatrix}
S_{11}&S_{12}\\
S_{21}&S_{22}\\
\end{bmatrix} \\
&= \begin{bmatrix}
-R^T_{21}G_m^T+G_nS_{21}&R^T_{11}G_m+G_nS_{22}\\
-R^T_{22}G_m^T-G_n^TS_{11}&R^T_{12}G_m-G_n^TS_{12}\\
\end{bmatrix},
\end{align*}
}
where $S_{ij}$ and $R_{ij}$ , $ i,j=1,2$ satisfy (\ref{spiv6}).

We reduce each pair of blocks independently.
First we reduce $B_{11}$.
Using the equality $R^T_{21}F_m^T=F_nS_{21}$ we
obtain that
\begin{equation*}
{\small
S_{21}=
\begin{bmatrix}
&Q&\\
a_{1}& \ldots &a_{m}\\
\end{bmatrix},
R^T_{21}=
\begin{bmatrix}
&&b_1\\
Q&& \vdots\\
&&b_n\\
\end{bmatrix}}, \text{where $Q=[q_{ij}]$ is any $n$-by-$m$ matrix.}
\end{equation*}
Therefore
{\small
\begin{multline}
\Delta B_{11}=-R^T_{21}G^T_m+G_nS_{21}= -
\begin{bmatrix}
&&b_1\\
Q&& \vdots\\
&&b_n\\
\end{bmatrix}
G_m^T+G_n
\begin{bmatrix}
&Q&\\
a_{1}& \ldots &a_{m}\\
\end{bmatrix}\\
=\begin{bmatrix}
q_{21}-q_{12}&q_{22}-q_{13}&  \ldots & q_{2,m-1}-q_{1m}& q_{2m}-b_1\\
q_{31}-q_{22}&q_{32}-q_{23}&  \ldots & q_{3,m-1}-q_{2m}& q_{3m}-b_2\\
q_{41}-q_{32}&q_{42}-q_{33}&  \ldots & q_{4,m-1}-q_{3m}& q_{4m}-b_3\\
\hdotsfor{5}\\
q_{n1}-q_{n-1,2}&q_{n2}-q_{n-1,3}&  \ldots & q_{n,m-1}-q_{n-1,m}& q_{nm}-b_{m-1}\\
a_1-q_{n2}&a_2-q_{n3}&  \ldots & a_{n-1}-q_{nm}& a_n-b_m\\
\end{bmatrix}.
\end{multline}
}
We can set each anti-diagonal of $B_{11}$ to zero independently
by adding the corresponding anti-diagonal of $\Delta B_{11}$.
Thus we can reduce $B_{11}$ by adding $\Delta B_{11}$ to zero.

Now to the pair $(A_{12},B_{12})$:
To preserve $A_{12},$  we take $R_{11}$ and $S_{22}$ such that
$R^T_{11}F_m=-F_nS_{22}$ thus
{\small
\begin{equation*}
S_{22}=
\begin{bmatrix}
&&&0\\
&-R_{11}^T&& \vdots\\
&&&0\\
b_1&\ldots&b_m&b_{m+1}\\
\end{bmatrix},
\end{equation*}
}
where $R_{11}^T$ is any $n$-by-$m$ matrix.
Thus
{\small
\begin{align*}
\Delta B_{12}&=R^T_{11}G_m+G_nS_{22}=
R^T_{11}G_m+G_n
\begin{bmatrix}
&&&0\\
&-R_{11}^T&& \vdots\\
&&&0\\
b_1&\ldots&b_m&b_{m+1}\\
\end{bmatrix} \\
&=\begin{bmatrix}
-r_{21}& r_{11}-r_{22}& r_{12}-r_{23}&\ldots & r_{1,m-1}-r_{2m}& r_{1m}\\
-r_{31}& r_{21}-r_{32}& r_{22}-r_{33}&\ldots & r_{2,m-1}-r_{3m}& r_{2m}\\
-r_{41}& r_{31}-r_{42}& r_{32}-r_{43}&\ldots & r_{3,m-1}-r_{4m}& r_{3m}\\
\hdotsfor{6}\\
-r_{n1}& r_{n-1,1}-r_{n2}& r_{n-1,2}-r_{n3}&\ldots & r_{n-1,m-1}-r_{nm}& r_{n-1m}\\
b_{1}& r_{n1}+b_{2}& r_{n2}+b_{3}&\ldots & r_{n,m-1}+b_{m}& r_{nm}+b_{m+1}\\
\end{bmatrix}.
\end{align*}
}
If $m+1 \geq n$ then we can set $B_{12}$ to zero
by adding $\Delta B_{12}$. If $n > m+1$ then we cannot set $B_{12}$ to zero. Then we reduce it diagonal-wise starting from the top-right corner. By adding the first $m$ and the last $m+1$ diagonals of $\Delta B_{12}$ we set the corresponding
diagonals of $B_{12}$ to zeros. We can set the remaining $n-m-1$
diagonals of $B_{12}$ to zeros, except the last element of each of them.
Hence $(A_{12},B_{12})$ is reduced to $(0,0^{\boxminus T}_{m+1,n})$ by
adding $\Delta (A_{12}, B_{12})$.

$(A_{21},B_{21})$ is reduced in the same way (up to the transposition) as $(A_{12},B_{12})$. Hence it
can be reduced to the form $(0,0^{\boxminus}_{n+1,m})$.

Consider $(A_{22},B_{22})$. We reduce $A_{22}$
to the form $0_*$ by adding $\Delta A_{22}=R^T_{12}F_m-F_n^TS_{12}$.
To preserve $A_{22},$ we must hereafter take $R_{12}$ and $S_{12}$
such that $R^T_{12}F_m=F_n^TS_{12}$ thus
\begin{equation*}
{\small
R^T_{12}=
\begin{bmatrix}
&Q&\\
0& \ldots &0\\
\end{bmatrix},
S_{12}=
\begin{bmatrix}
&&0\\
Q&& \vdots\\
&&0\\
\end{bmatrix}}, \text{ where $Q=[q_{ij}]$ is any $n$-by-$m$ matrix.}
\end{equation*}
Therefore,
{\small
\begin{align*}
\Delta B_{22}&=R^T_{12}G_m-G^T_nS_{12}=
\begin{bmatrix}
&Q&\\
0& \ldots &0\\
\end{bmatrix}
G_m-G_n^T
\begin{bmatrix}
&&0\\
Q&& \vdots\\
&&0\\
\end{bmatrix} \\&=
\begin{bmatrix}
0& q_{11}& q_{12}&  \ldots & q_{1,m-1}& q_{1m} \\
-q_{11}& q_{21}-q_{12}& q_{22}-q_{13}&  \ldots& q_{2,m-1}-q_{1m}& q_{2m} \\
-q_{21}& q_{31}-q_{22}& q_{32}-q_{23}&  \ldots& q_{3,m-1}-q_{2m}& q_{3m} \\
\hdotsfor{6}\\
-q_{n-1,1}& q_{n1}-q_{n-1,2}& q_{n2}-q_{n-1,3}&  \ldots& q_{n,m-1}-q_{n-1,m}& q_{nm} \\
-q_{n1}& -q_{n2}& -q_{n3}&  \ldots& -q_{nm}& 0
\end{bmatrix}.
\end{align*}}\noindent
By adding $\Delta B_{22}$, we can set each element of
$B_{22}$ to zero except the elements in the
first column and the last row {(or, alternatively,} the elements in the first row and the last column{)}.

Summing up the results, we have that
${\cal D}(\mathcal L_n, \mathcal L_m)$ is of the form \eqref{lsiu3}.

\subsection{Off-diagonal
blocks of $\cal D$ that
correspond to summands of
$(A,B)_{\text{can}}$
of different types}

Finally, we verify the
condition (ii) of
Lemma \ref{thekd} for
off-diagonal blocks of
$\cal D$ defined in
Theorem
\ref{teo2}(iii); the
diagonal blocks of
their horizontal and
vertical strips
contain summands of
$(A,B)_{\text{can}}$ of different
types.

\subsubsection{Pairs of
blocks ${\cal D}(\mathcal H_n(\lambda),\mathcal K_m)$ }
\label{sub7}

Due to Lemma
\ref{thekd}(ii), it
suffices to prove that
each group of four matrices $((A,B),(-A^T,-B^T))$
can be reduced to
exactly one group of
the form \eqref{kut}
by adding
\[
(R^T\mathcal K_m+ \mathcal H_n(\lambda) S,
S^T\mathcal H_n(\lambda)
+\mathcal K_mR),\quad R\in
 {\mathbb C}^{2m\times 2n},\ S\in
 {\mathbb C}^{2n\times 2m}.
\]
Obviously, if we reduce $(A,B)$ then the second pair will be
reduced automatically. We have 
\begin{equation*}
{\small
\begin{split}
&\Delta (A,B)=R^T\mathcal K_m
+\mathbb \mathcal H_m(\lambda) S \\ &=  \bigg( R^T
\begin{bmatrix}
0&J_m(0)\\
-J_m(0)^T&0\\
\end{bmatrix}
+
\begin{bmatrix}
0&I_n\\
-I_n&0\\
\end{bmatrix}
S, R^T
\begin{bmatrix}
0&I_m\\
-I_m&0\\
\end{bmatrix}
+
\begin{bmatrix}
0&J_n(\lambda)\\
-J_n(\lambda)^T&0\\
\end{bmatrix}
S \bigg).
\end{split}
}
\end{equation*}
It is clear that we can set $A$ to zero. To preserve $A,$ we
 must hereafter take $R=[R_{ij}]_{i,j=1}^2$ and $S=[S_{ij}]_{i,j=1}^2$ such that
{\small
\[
R^T
\begin{bmatrix}
0&J_m(0)\\
-J_m(0)^T&0\\
\end{bmatrix}
+
\begin{bmatrix}
0&I_n\\
-I_n&0\\
\end{bmatrix}
S=0,
\]
}
or equivalently 
{\small
\begin{equation*}
S=
\begin{bmatrix}
-R^T_{22}J_m(0)^T&R^T_{12}J_m(0)\\
R^T_{21}J_m(0)^T&-R^T_{11}J_m(0)\\
\end{bmatrix}.
\end{equation*}
}
Therefore $B=[B_{ij}]_{i,j=1}^2$ is reduced by adding
{\small
\begin{align*}
\Delta B & = \begin{bmatrix}
\Delta B_{11}&\Delta B_{12}\\
\Delta B_{21}&\Delta B_{22}\\
\end{bmatrix} \\ &=
\begin{bmatrix}
R^T_{11}&R^T_{21}\\
R^T_{12}&R^T_{22}\\
\end{bmatrix}
\begin{bmatrix}
0&I_m\\
-I_m&0\\
\end{bmatrix}
+
\begin{bmatrix}
0&J_n(\lambda)\\
-J_n(\lambda)^T&0\\
\end{bmatrix}
\begin{bmatrix}
-R^T_{22}J_m(0)^T&R^T_{12}J_m(0)\\
R^T_{21}J_m(0)^T&-R^T_{11}J_m(0)\\
\end{bmatrix}\\
& = \begin{bmatrix}
-R^T_{21}+J_n(\lambda)R^T_{21}J_m(0)^T&R^T_{11}-J_n(\lambda)R^T_{11}J_m(0)\\
-R^T_{22}+J_n(\lambda)^TR^T_{22}J_m(0)^T&R^T_{12}-J_n(\lambda)^TR^T_{12}J_m(0)\\
\end{bmatrix}.
\end{align*}
}
The block $B_{11}$ is reduced to zero by adding
{\small
\begin{align*}
\Delta B_{11}& =-R^T_{21}+J_n(\lambda)R^T_{21}J_m(0)^T \\ &=
\begin{cases}
-r_{ij}+\lambda r_{i,j+1}+r_{i+1,j+1}&  \text{if }1 \leq i \leq n-1,1 \leq j \leq m-1,\\
-r_{ij}+\lambda r_{i,j+1}&  \text{if }1 \leq j \leq m-1,i=n, \\
-r_{ij}&  \text{if }1 \leq i \leq n,j=m,
\end{cases}
\end{align*}
}
because it results in a square system of $nm$ equations that has a solution.

The reduction of the other blocks follows from above since 
{\small
\begin{align*}
R^T_{11}-J_n(\lambda)R^T_{11}J_m(0)&=-R^T_{21}Z+J_n(\lambda)R^T_{21}ZZJ_m(0)^TZ,\\
-R^T_{22}+J_n(\lambda)^TR^T_{22}J_m(0)^T&=-ZR^T_{21}Z+ZJ_n(\lambda)ZZR^T_{21}ZZJ_m(0)^TZ,\\
R^T_{12}-J_n(\lambda)^TR^T_{12}J_m(0)&=-ZR^T_{21}+ZJ_n(\lambda)ZZR^T_{21}J_m(0)^T, 
\end{align*}
}
where the matrices $Z$ (see \eqref{zzz}) are of the corresponding sizes.

Altogether, we have that ${\cal D}(\mathcal H_n(\lambda), \mathcal K_m)$ is zero.

\subsubsection{Pairs of
blocks ${\cal D}(\mathcal H_n(\lambda), \mathcal L_m)$ }
\label{sub8}

Due to Lemma
\ref{thekd}(ii), it
suffices to prove that
each group of four matrices $((A,B),(-A^T,-B^T))$
can be reduced to the group of the form \eqref{hnlm} by adding
\[
(R^T\mathcal L_m
+\mathcal H_n(\lambda)S, S^T\mathcal H_n(\lambda)
+ \mathcal L_mR),\quad S\in
 {\mathbb C}^{2n\times 2m+1},\ R\in
 {\mathbb C}^{2m+1\times 2n}.
\]
Obviously, if we only reduce $(A,B)$, then $(-A^T,-B^T)$ will be
 reduced automatically. We have 
 {\small
\begin{multline*}
\Delta (A,B)=R^T\mathcal L_m
+ \mathcal H_n(\lambda)S \\ =\left(R^T
\begin{bmatrix}
0&F_m\\
-F_m^T&0\\
\end{bmatrix}
+
\begin{bmatrix}
0&I_n\\
-I_n&0\\
\end{bmatrix}
S,R^T
\begin{bmatrix}
0&G_m\\
-G_m^T&0\\
\end{bmatrix}
+
\begin{bmatrix}
0&J_n(\lambda)\\
-J_n(\lambda)^T&0\\
\end{bmatrix}
S\right) .
\end{multline*}
}
It is easy to check that we can set $A$ to zero. To preserve
$A,$ we must hereafter take $R=[R_{ij}]_{i,j=1}^2$ and $S=[S_{ij}]_{i,j=1}^2$ such that
{\small
\[
R^T
\begin{bmatrix}
0&F_m\\
-F_m^T&0\\
\end{bmatrix}
+
\begin{bmatrix}
0&I_n\\
-I_n&0\\
\end{bmatrix}
S=0, \text{ \normalsize{or equivalently}  }
S=
\begin{bmatrix}
-R^T_{22}F_m^T&R^T_{12}F_m\\
R^T_{21}F_m^T&-R^T_{11}F_m\\
\end{bmatrix}.
\]
}
Thus $B=[B_{ij}]_{i,j=1}^2$ is reduced by adding
{\small
\begin{align*}
\Delta B &=
\begin{bmatrix}
\Delta B_{11}&\Delta B_{12}\\
\Delta B_{21}&\Delta B_{22}\\
\end{bmatrix} \\�&=
\begin{bmatrix}
R^T_{11}&R^T_{21}\\
R^T_{12}&R^T_{22}\\
\end{bmatrix}
\begin{bmatrix}
0&G_m\\
-G^T_m&0\\
\end{bmatrix} 
+
\begin{bmatrix}
0&J_n(\lambda)\\
-J_n(\lambda)^T&0\\
\end{bmatrix}
\begin{bmatrix}
-R^T_{22}F_m^T&R^T_{12}F_m\\
R^T_{21}F_m^T&-R^T_{11}F_m\\
\end{bmatrix}\\
&= \begin{bmatrix}
-R^T_{21}G^T_m+J_n(\lambda)R^T_{21}F_m^T&R^T_{11}G_m-J_n(\lambda)R^T_{11}F_m\\
-R^T_{22}G^T_m+J_n(\lambda)^TR^T_{22}F_m^T&R^T_{12}G_m-J_n(\lambda)^TR^T_{12}F_m\\
\end{bmatrix}.
\end{align*}
}
First, adding 
{\small
\begin{multline*}
\Delta B_{11}=-R^T_{21}G^T_m+J_n(\lambda)R^T_{21}F_m^T= \\
\begin{bmatrix}
-r_{12}+\lambda r_{11}+r_{21}&-r_{13}+ \lambda r_{12}+r_{22}&  \ldots & -r_{1,m+1}+ \lambda r_{1m}+r_{2m}\\
-r_{22}+\lambda r_{21}+r_{31}&-r_{23}+ \lambda r_{22}+r_{32}&  \ldots & -r_{2,m+1}+ \lambda r_{2m}+r_{3m}\\
\hdotsfor{4}\\
-r_{n-1,2}+\lambda r_{n-1,1}+r_{n1}&-r_{n-1,3}+ \lambda r_{n-1,2}+r_{n2}&  \ldots & -r_{n-1,m+1}+ \lambda r_{n-1,m}+r_{nm}\\
-r_{n2}+\lambda r_{n1}&-r_{n3}+ \lambda r_{n2}&  \ldots & -r_{n,m+1}+ \lambda r_{nm}\\
\end{bmatrix},
\end{multline*}}\noindent 
we can set $B_{11}$ to zero as follows. 
For the last ($n$-th) row of  $B_{11}$ we have the following system of equations 
{\small
\begin{equation} \label{syst999}
\left[
\begin{matrix}
\lambda&-1&&&\\
&\lambda&-1&&\\
&&\ddots&\ddots&\\
&&&\lambda&-1
\end{matrix}
\right] 
\left[
 \begin{matrix}
r_{n1} \\
r_{n2} \\
\vdots \\
r_{nm} \\
r_{n,m+1} 
\end{matrix}
\right]
=
\left[
 \begin{matrix}
b_1 \\
b_2 \\
\vdots \\
b_m
\end{matrix}
\right]
\end{equation}
}
which has a solution. For the $(n-1)$-th row we have
{\small
\begin{equation} \label{system57}
\left[
\begin{matrix}
\lambda&-1&&&\\
&\lambda&-1&&\\
&&\ddots&\ddots&\\
&&&\lambda&-1
\end{matrix}
\right] 
\left[
 \begin{matrix}
r_{n-1,1} \\
r_{n-1,2} \\
\vdots \\
r_{n-1,m} \\
r_{n-1,m+1} 
\end{matrix}
\right]
=
\left[
 \begin{matrix}
b_1 \\
b_2 \\
\vdots \\
b_m
\end{matrix}
\right] - 
\left[
 \begin{matrix}
r_{n1} \\
r_{n2} \\
\vdots \\
r_{nm}
\end{matrix}
\right].
\end{equation}
}
The variables ${r_{n1},r_{n2}, \ldots , r_{nm}}$ are known from \eqref{syst999}, thus \eqref{system57} becomes a system of the type \eqref{syst999} and the system \eqref{system57} has a solution. Repeating this reduction to every row from the bottom to the top, we set $B_{11}$ to zero. 

The block $B_{21}$ is reduced like the block $B_{11}$ and thus we omit this verification.

Now we turn to the reduction of $B_{12}$ and $B_{22}$. It suffices to consider only $B_{12}$. We have
{\small
\begin{multline*}
\Delta B_{12}=R^T_{11}G_m-J_n(\lambda)R^T_{11}F_m\\= 
\begin{bmatrix}
-\lambda r_{11}-r_{21}& r_{11}-\lambda r_{12}-r_{22}&  \ldots& r_{1,m-1}-\lambda r_{1m}-r_{2m}& r_{1m} \\
-\lambda r_{21}-r_{31}& r_{21}-\lambda r_{22}-r_{32}&  \ldots& r_{2,m-1}-\lambda r_{2m}-r_{3m}& r_{2m} \\
\hdotsfor{5}\\
-\lambda r_{n-1,1}-r_{n1}& r_{n-1,1}-\lambda r_{n-1,2}-r_{n2}&  \ldots& r_{n-1,m-1}-\lambda r_{n-1,m}-r_{nm}& r_{n-1,m} \\
-\lambda r_{n1}& r_{n1}-\lambda r_{n2}&  \ldots& r_{n,m-1}-\lambda r_{nm}& r_{nm} \\
\end{bmatrix}.
\end{multline*}
}
Adding $\Delta B_{12}$ we reduce $B_{12}$ to the form $0^{\leftarrow}$.

Summing up the results for all the blocks, we have that ${\cal D}(\mathcal H_n(\lambda), \mathcal L_m)$ is
equal to \eqref{hnlm}.

\subsubsection{Pairs of
blocks ${\cal D}(\mathcal K_n, \mathcal L_m)$ }
\label{sub9}

Due to Lemma
\ref{thekd}(ii), it
suffices to prove that
each group of four matrices $((A,B),(-A^T,-B^T))$
can be reduced to
the group of the form \eqref{ktlm}
by adding
\[
(R^T\mathcal L_m
+\mathcal K_nS, S^T
\mathcal K_n+ \mathcal L_mR),\quad S\in
 {\mathbb C}^{2n\times 2m+1},\ R\in
 {\mathbb C}^{2m+1\times 2n}.
\]
As in the previous sections, we reduce only $(A,B)$ and $(-A^T,-B^T)$ is reduced automatically. We have 
{\small
\begin{multline*}
\Delta (A,B)=R^T\mathcal L_m+\mathcal K_n S \\ =\left(R^T
\begin{bmatrix}
0&F_m\\
-F_m^T&0\\
\end{bmatrix}
+
\begin{bmatrix}
0&J_n(0)\\
-J_n(0)^T&0\\
\end{bmatrix}
S,R^T
\begin{bmatrix}
0&G_m\\
-G_m^T&0\\
\end{bmatrix}
+
\begin{bmatrix}
0&I_n\\
-I_n&0\\
\end{bmatrix}
S \right) .
\end{multline*}}\noindent
It is clear that we can set $B$ to zero. To preserve $B,$
we must hereafter take $R=[R_{ij}]_{i,j=1}^2$ and $S=[S_{ij}]_{i,j=1}^2$ such that
{\small
\[
R^T
\begin{bmatrix}
0&G_m\\
-G_m^T&0\\
\end{bmatrix}
+
\begin{bmatrix}
0&I_n\\
-I_n&0\\
\end{bmatrix}
S=0,
\text{ \normalsize{or equivalently} }
S=
\begin{bmatrix}
-R^T_{22}G_m^T&R^T_{12}G_m\\
R^T_{21}G_m^T&-R^T_{11}G_m\\
\end{bmatrix}.
\]
}
Hence $A=[A_{ij}]_{i,j=1}^2$ is reduced by adding
{\small
\begin{align*}
\Delta A &=
\begin{bmatrix}
\Delta A_{11}&\Delta A_{12}\\
\Delta A_{21}&\Delta A_{22}\\
\end{bmatrix} \\ &=
\begin{bmatrix}
R^T_{11}&R^T_{21}\\
R^T_{12}&R^T_{22}\\
\end{bmatrix}
\begin{bmatrix}
0&F_m\\
-F^T_m&0\\
\end{bmatrix} 
+
\begin{bmatrix}
0&J_n(0)\\
-J_n(0)^T&0\\
\end{bmatrix}
\begin{bmatrix}
-R^T_{22}G_m^T&R^T_{12}G_m\\
R^T_{21}G_m^T&-R^T_{11}G_m\\
\end{bmatrix}\\
&= \begin{bmatrix}
-R^T_{21}F^T_m+J_n(0)R^T_{21}G_m^T&R^T_{11}F_m-J_n(0)R^T_{11}G_m\\
-R^T_{22}F^T_m+J_n(0)^TR^T_{22}G_m^T&R^T_{12}F_m-J_n(0)^TR^T_{12}G_m\\
\end{bmatrix}.
\end{align*}
}
First we reduce the block $A_{11}$ ($A_{21}$ is reduced in the same way). We have 
{\small
\begin{multline*}
\Delta A_{11}=-R^T_{21}F^T_m+J_n(0)R^T_{21}G_m^T\\=
\begin{bmatrix}
-r_{11}+r_{22}&-r_{12}+r_{23}&-r_{13}+r_{24}&  \ldots & -r_{1m}+ r_{2,m+1}\\
-r_{21}+r_{32}&-r_{22}+r_{33}&-r_{23}+r_{34}&  \ldots & -r_{2m}+ r_{3,m+1}\\
\hdotsfor{5}\\
-r_{n-1,1}+r_{n2}&-r_{n-1,2}+r_{n3}&-r_{n-1,3}+r_{n4}&  \ldots & -r_{n-1,m}+ r_{n,m+1}\\
-r_{n1}&-r_{n2}&-r_{n4}&  \ldots & -r_{nm}\\
\end{bmatrix},
\end{multline*}}\noindent 
and thus we reduce each diagonal of $A_{11}$ independently.
For each of the first $m$ diagonals, starting from the bottom-left corner, we have a system of type 
\eqref{systsov}
which has a solution, and for the remaining diagonals we have the system of type \eqref{4sdf}
which has a solution too.
Thus adding $\Delta A_{11}$ we set $A_{11}$ to zero.

Last, we reduce the blocks $A_{12}$ and $A_{22}$ and it is enough to consider
only $A_{12}$. We have 
{\small
\begin{align*}
\Delta A_{12}&=R^T_{11}F_m-J_n(0)R^T_{11}G_m\\ &=
\begin{bmatrix}
 r_{11}&  r_{12}-r_{21}&  r_{13}-r_{22} & \ldots&  r_{1m}-r_{2,m-1}& -r_{2m} \\
 r_{21}&  r_{22}-r_{31}&  r_{13}-r_{32} & \ldots&  r_{2m}-r_{3,m-1}& -r_{2m} \\
\hdotsfor{6}\\
 r_{n-1,1}&  r_{n-1,2}-r_{n1}&  r_{n-1,3}-r_{n2} & \ldots& r_{n-1,m}-r_{n,m-1}& -r_{nm} \\
 r_{n1}&  r_{n2}& r_{n3}&  \ldots& r_{nm}& 0 \\
\end{bmatrix}.
\end{align*}
}
Adding $\Delta A_{12}$ we reduce $A_{12}$ to the form $0^{\rightarrow}$.

Summing up the results for all blocks we have that ${\cal D}(\mathcal K_n, \mathcal L_m)$ is equal
to \eqref{ktlm}.

\section*{Acknowledgements}

The author is thankful to Bo K\r{a}gstr\"{o}m and Vladimir V. Sergeichuk for their constructive comments and discussions on the manuscript. {The author also thanks to the anonymous referees for the helpful remarks and suggestions.}

This is an extended version of a part of the author's Master Thesis \cite{dm1}, written under the supervision of Vladimir V. Sergeichuk at the Kiev National University.

The work was supported by the Swedish Research 
Council (VR) under grant E0485301, and by eSSENCE, a strategic collaborative e-Science programme funded by the Swedish Research Council.


{\footnotesize
\begin{thebibliography}{9999}
%

\bibitem{arn}
V.I. Arnold, On
matrices depending on
parameters,
Russian Math. Surveys,
26 (2) (1971) 29--43.

\bibitem{arn2}
V.I. Arnold, Lectures on
bifurcations in
versal families,
Russian Math. Surveys,
27 (5) (1972) 54--123.

\bibitem{arn3}
V.I. Arnold, 
Geometrical methods in
the theory of ordinary
differential
equations,
Springer-Verlag, New
York, 1988.

\bibitem{b} L. Batzke, Generic low rank perturbations of structured regular matrix pencils and structured matrices, Doctoral Thesis, TU Berlin, 2015.

\bibitem{mehr1} R. Byers, V. Mehrmann, and H. Xu, A structured staircase algorithm for skew-symmetric/symmetric pencils, Electron. Trans. Numer. Anal. 26 (2007) 1--33. 

\bibitem{multi} T.J. Bridges, S. Reich, Multi-symplectic integrators: numerical schemes for Hamiltonian PDEs that conserve symplecticity, Physics Letters A, 284(4--5), (2001) 184--193.

\bibitem{mehr2}T. Br\"{u}ll and V. Mehrmann, STCSSP: A FORTRAN 77 routine to compute a structured staircase form for a (skew-)symmetric/(skew-)symmetric matrix pencil, Preprint 31, Instituts f\"{u}r Mathematik Technische Universit\"{a}lt, 2007.

\bibitem{dt_d_1} F. De Ter\'an and F.M. Dopico, The solution of the equation $XA + AX^T  = 0$ and its application to the theory of orbits, Linear Algebra Appl. 434 (2011) 44--67.

\bibitem{dt_d_2} F. De Ter\'an and F.M. Dopico, The equation $XA + AX^*= 0$ and the dimension of *congruence orbits, Electr. J. Linear Algebra, 22 (2011) 448--465.

\bibitem{dm1} A.R. Dmytryshyn, Miniversal Deformations of Pairs of Skew-symmetric Forms,  Master Thesis, Kiev National University, Kiev, 2010.

\bibitem{dm2} A.R. Dmytryshyn,
    Miniversal deformations of pairs of
    symmetric forms, Manuscript, 2011,
    arXiv:1104.2530.

\bibitem{sstool} A. Dmytryshyn, S. Johansson, and B. K\r{a}gstr\"{o}m, Codimension computations of congruence orbits of matrices, skew-symmetric and symmetric matrix pencils using Matlab, Technical report UMINF 13.18, Dept. of Computing Science, Ume\r{a} University, Sweden, 2013.  

\bibitem{bfg} A. Dmytryshyn, V. Futorny, B. K\r{a}gstr\"{o}m, L. Klimenko, and V.V. Sergeichuk, Change of the congruence
canonical form of 2-by-2 and 3-by-3 matrices under perturbations and bundles of matrices under congruence, 
Linear Algebra Appl., 469 (2015) 305--334.

\bibitem{bilin}
A.R. Dmytryshyn, V. Futorny, and V.V. Sergeichuk, 
Miniversal deformations of matrices of bilinear
forms, Linear Algebra Appl., 436(7)  (2012) 2670--2700.


\bibitem{sf} A. Dmytryshyn, V. Futorny, and V.V. Sergeichuk, Miniversal deformations of matrices under *congruence and reducing transformations, 
Linear Algebra Appl., 446 (2014) 388--420.

\bibitem{skewstr} A. Dmytryshyn and B. K\r{a}gstr\"{o}m, Orbit closure hierarchies of skew-symmetric matrix pencils, SIAM J. Matrix Anal. Appl., 35(4) (2014) 1429--1443.

\bibitem{dmsystss}  A. Dmytryshyn, B. K\r{a}gstr\"{o}m, and V.V. Sergeichuk, Skew-symmetric matrix pencils: Codimension counts and the solution of a pair of matrix equations, Linear Algebra Appl., 438(8) (2013) 3375--3396.

\bibitem{dmsysts}  A. Dmytryshyn, B. K\r{a}gstr\"{o}m, and V.V. Sergeichuk, Symmetric matrix pencils: Codimension counts and the solution of a pair of matrix equations, Electron. J. Linear Algebra, 27 (2014) 1--18.


\bibitem{bo1}
A. Edelman, E.
Elmroth, and B.
K\r{a}gstr\"{o}m, A
geometric approach to
perturbation theory of
matrices and matrix
pencils. Part I:
Versal deformations, SIAM J. Matrix
Anal. Appl., 18(3)
(1997) 653--692.

\bibitem{bo2}
A. Edelman, E.
Elmroth, and B.
K\r{a}gstr\"{o}m, A
geometric approach to
perturbation theory of
matrices and matrix
pencils. Part II: A
stratification-enhanced
staircase algorithm, SIAM J. Matrix
Anal. Appl., 20 (1999)
667--669.

\bibitem{bo4}
E. Elmroth, S. Johansson, and B.
K\r{a}gstr\"{o}m,
Stratification of controllability and observability
pairs theory and use in applications, SIAM J. Matrix Anal. Appl.,
 31(2) (2009) 203--226.

\bibitem{f_k_s} V. Futorny, L.
    Klimenko, and V.V. Sergeichuk,
    Change of the *congruence
    canonical form of 2-by-2 matrices
    under perturbations, Electr.
    J. Linear Algebra 27 (2014)
    146--154.

\bibitem{sgp}
M.I. Garcia-Planas and V.V. Sergeichuk,
Simplest miniversal
deformations of
matrices, matrix
pencils, and
contragredient matrix
pencils, Linear
Algebra Appl.,
302--303 (1999)
45--61.

\bibitem{hor_John}
R.A. Horn and C.R. Johnson, Matrix
Analysis, Cambridge
U. P., Cambridge,
1985.

\bibitem{toolbox}
P. Johansson, Matrix canonical structure toolbox, Technical report UMINF~06.15, Dept. of Computing Science, Ume\r{a} University, Sweden, 2006.

\bibitem{k} A.G. Kurosh,  Higher algebra, Mir Publishers, 1972. 

\bibitem{bobook}
  B.~K{\aa}gstr\"{o}m, S.~Johansson, and P.~Johansson, 
  StratiGraph Tool: Matrix Stratification in Control Applications.
  In L.~Biegler, S.~L. Campbell, and V.~Mehrmann, editors, Control and Optimization with Differential-Algebraic
    Constraints, chapter~5. SIAM Publications, 2012.

\bibitem{robot} P.Y. Li, R. Horowitz, Passive Velocity Field Control of Mechanical Manipulators, IEEE Transactions on Robotics and Automation, 15 (4) (1999) 751 --763.

\bibitem{lancrod} A. Gohberg, P. Lancaster, and L. Rodman, Invariant Subspaces of Matrices with Applications, Vol. 51. SIAM, 1986.

\bibitem{ref4} A.A. Mailybaev, Transformation of families of matrices to normal forms and its application to stability theory, SIAM J. Matrix Anal. Appl., 21 (2000) 396--417.

\bibitem{maltcev}
A.I. Mal'cev, Foundations of linear algebra. Translated from the Russian by Thomas Craig Brown, J. B. Roberts, ed.,  W. H. Freeman \& Co., San Francisco, 1963.


\bibitem{m_m_t} D.S. Mackey, N. Mackey, and F. Tisseur, Polynomial Eigenvalue Problems: Theory, Computation, and Structure, in P. Benner et. al., Numerical Algebra, Matrix Theory, Differential-Algebraic Equations and Control Theory, chapter~12, Springer, 2015.

\bibitem{m4skew} D.S. Mackey, N. Mackey, C. Mehl, and V. Mehrmann, Skew-symmetric matrix polynomials and their Smith forms, Linear
Algebra Appl., 438(12) (2013) 4625--4653.


\bibitem{biham} P.J. Olver,  Canonical forms for compatible biHamiltonian systems. In I.~Antoniou and F.~Lambert, editors, Solitons and Chaos, pp. 171--179, Springer-Verlag, New York, 1991.

\bibitem{mo1} L. Rodman, Comparison of congruences and strict equivalences for real, complex, and quaternionic matrix pencils with symmetries,. Electron. J. Linear Algebra, 16 (2007), 248--283.


\bibitem{s_76} R. Scharlau, Paare alternierender Formen, Math. Z. 147 (1976) 13--19.

\bibitem{thom}
R.C. Thompson,
  Pencils of complex and
real symmetric and
skew matrices, Linear Algebra Appl.,
147 (1991) 323--371.


\end{thebibliography}
}
\end{document}